\documentclass{birkjour}

\newtheorem{theorem}{Theorem}[section]
\newtheorem{lemma}[theorem]{Lemma}

\newtheorem{assumption}[theorem]{Assumption}
\newtheorem{corollary}[theorem]{Corollary}
\newtheorem{definition}[theorem]{Definition}
\newtheorem{example}[theorem]{Example}
\newtheorem{remark}[theorem]{Remark}
\newtheorem{hypothesis}[theorem]{Hypothesis}
\renewcommand{\theequation}{\arabic{section}.\arabic{equation}}

\newcommand{\Tr}{\mathop{\mathrm{Tr}}}
\renewcommand{\d}{\/\mathrm{d}\/}
\def\s{^{\star}}

\def\u{u^{n, \varepsilon}}
\def\ue{u^{\varepsilon}}

\def\ve{v^{\varepsilon}}
\def\we{w^{\varepsilon}}
\def\e{\varepsilon}
\def\sni{\smallskip\noindent}
\def\t{t\wedge\tau_N}
\def\T{T\wedge\tau_N}

\begin{document}

\title{Shell Model of Turbulence Perturbed by L\'{e}vy Noise}

\author[Utpal Manna]{Utpal Manna}

\address{%
School of Mathematics\\ 
Indian Institute of Science Education and Research (IISER) Thiruvananthapuram\\
Thiruvananthapuram 695016\\
Kerala, INDIA}

\email{manna.utpal@iisertvm.ac.in}

\author[Manil T. Mohan]{Manil T. Mohan}

\address{%
School of Mathematics\\ 
Indian Institute of Science Education and Research (IISER) Thiruvananthapuram\\
Thiruvananthapuram 695016\\
Kerala, INDIA}

\email{manil@iisertvm.ac.in}



\begin{abstract}
In this work we prove the existence and uniqueness of the strong
solution of the shell model of turbulence perturbed by L\'{e}vy
noise. The local monotonicity arguments have been exploited in the
proofs.
\end{abstract}

\maketitle\setcounter{equation}{0}

\maketitle
\section{Introduction}

Stochastic partial differential equations driven by jump processes
gain attention quiet recently due to its important applications in
Mathematical Physics (see, Albeverio et al. \cite{ARW} and
Shlesinger et al. \cite{SZF}) and also in Biomathematics (see West
\cite{We}). Detailed literature on this subject can be found in the
books by Applebaum \cite{Ap}, Ikeda and Watanabe \cite{IW}, and
Peszat and Zabczyk \cite{PZ} (and references therein). In the last
few years several interesting results have been established. To name
a few, de Acosta \cite{Ac1, Ac} first studied the large deviations
for L\'{e}vy processes on Banach spaces and large deviations for
solutions of stochastic differential equations driven by Poisson
measures; Albeverio et al. \cite{Al} proved the existence and
uniqueness for solutions of parabolic SPDEs driven by Poisson random
measures; absolute continuity of the law of the solutions of
parabolic SPDEs driven by Poisson random measures was proved by
Fournier \cite{Fo} using techniques from Malliavin calculus;
Hausenblas \cite{Ha}, Mandrekar and R\"{u}diger \cite{MR} and
R\"{u}diger \cite{Ru} extensively studied the existence and
uniqueness of stochastic integral equations driven by L\'{e}vy noise
and compensated Poisson random measures on separable Banach spaces;
Mueller \cite{Mu} proved the short time existence for the solutions
(which is a minimal solution) of stochastic heat equation with
non-negative L\'{e}vy noise; R\"{o}ckner and Zhang \cite{RZ}
established the existence and uniqueness results for solutions of
stochastic evolution equations driven by L\'{e}vy noise and obtained
the large deviation principles for the additive L\'{e}vy noise case;
Zhao and Chao \cite{ZY} established the global existence and
uniqueness of the strong solution for 2D Navier-Stokes equations on
the torus perturbed by a L\'{e}vy process.

This work deals with an infinite dimensional shell model, a
mathematical turbulence model that received increasing attention in
recent years. Apparently there are only a few rigorous works on
infinite dimensional shell model, namely \cite{CLT}, and \cite{Ba}
one in the deterministic case and the other in the stochastic case
with additive noise respectively. In both of these works a
variational semigroup formulation has been introduced. In \cite{Ma1}
the existence and uniqueness of the strong solutions of the
stochastic shell model of turbulence perturbed by multiplicative
noise have been proved. The authors have also established a large
deviation principle for the solution of the shell model.  Our
present work deals with a more general stochastic model, with
L\'{e}vy noise: the proofs of existence and uniqueness of strong
solutions are considerably more difficult in this case.

In this paper, in the framework of Gelfand triple $V \subset H \cong
H'\subset V'$ (see section 3 for precise definitions), we consider
the following abstract form of the GOY model of turbulence with L\'{e}vy noise:
\begin{align*}
&\d u + \big[\nu Au + B(u, u)\big] \d t = f(t) \d t + \sqrt{\varepsilon}\sigma(t, u) \d W(t) + \varepsilon\int_Zg(u,z)\tilde{N}(dt,dz)\\
& u(0) = u_0,
\end{align*}
The operators $A$ and $B$ are defined in Section 3. $W(t)$ is an $H$-valued Wiener process with positive
symmetric trace class covariance operator $Q$. $\tilde{N}(dt, dz) =
N(dt, dz) - dt
\lambda(dz)$ is a compensated Poisson random measure (cPrm), where \\
$N(dt, dz)$ denotes the Poisson counting measure associated to
Poisson point process $p(t)$ on $Z$ and $\lambda(dz)$ is a
$\sigma$-finite measure on $(Z, \mathcal{B}(Z))$.

The main result of the paper is the following theorem. The spaces
$V, V', H, H_0$, $L_Q(H_0; H), \mathbb{H}^2_{\lambda}([0, T] \times
Z; H), \mathcal{D}([0, T]; H)$ which appear in the statement of this
theorem are defined in section 2.

\begin{theorem} [Main Theorem]
Let us consider the above stochastic GOY model of turbulence driven by L\'{e}vy
processes with the initial
condition $u_0(x)$. Let $u_0$ be $\mathcal{F}_0$ measurable and $E|u_0|^2<\infty$. Let $f \in L^2(0, T;
V')$. Assume that $\sigma$ and $g$ satisfy the following hypotheses
of joint continuity, Lipschitz condition and linear growth:
\begin{itemize}
    \item[(i)]  The function $\sigma \in C([0, T] \times V; L_Q(H_0; H))$, and $g\in \mathbb{H}^2_{\lambda}([0, T] \times Z; H)$.

    \item[(ii)]  For all $t \in (0, T)$, there exists a positive constant $K$ such that for all $u \in H$,
    $$|\sigma(t, u)|^2_{L_Q} + \int_{Z} |g(u, z)|^2_{H}\lambda(dz) \leq K(1 +|u|^2).$$

    \item[(iii)]  For all $t \in (0, T)$,  there exists a positive constant $L$ such that for all $u, v \in H$,
    $$|\sigma(t, u) - \sigma(t, v)|^2_{L_Q} + \int_{Z} |g(u, z)-g(v, z)|^2_{H}\lambda(dz)\leq L|u - v|^2.$$
\end{itemize}
Then there exist a unique adapted process $u(t, x, \omega)$ with
regularity
$$u \in L^2(\Omega; L^2(0, T; V) \cap \mathcal{D}(0, T; H))$$ satisfying the above stochastic GOY model in the weak sense.
\end{theorem}

The construction of the paper is as follows. In the next Section, we
give definitions, basic properties and It\^{o}'s formula for the
L\'{e}vy noise. In Section $3$, we describe the functional setting
and formulate the abstract stochastic shell model (namely GOY model)
when the noise coefficients are small. In section $4$, we prove
certain a-priori energy estimates with exponential weight. These
estimates together with the local monotonicity property of the sum
of the linear and non linear operators play a fundamental role to
prove the existence and uniqueness of the strong solution. The main result of this paper
as given in the above theorem has been proved in section $4$. 

\section{Preliminaries}

In this section definitions and basic properties of Hilbert space
valued Wiener processes and L\'{e}vy processes have been presented.
Most of the materials in this section have been borrowed from the
books by Da Prato and Zabczyk \cite{DaZ}, Applebaum \cite{Ap}, and
Peszat and Zabczyk \cite{PZ}. Interested readers may look into these
books for extensive study on the subject.
\begin{definition}
Let $H$ be a Hilbert space. A stochastic process $\{W(t)\}_{0\leq
t\leq T}$ is said to be an $H$-valued $\mathcal{F}_t$-adapted Wiener
process with covariance operator $Q$ if
\begin{enumerate}
\item [(i)] For each non-zero $h\in H$, $|Q^{1/2}h|^{-1} (W(t), h)$ is a standard one-dimensional Wiener process,
\item [(ii)] For any $h\in H, (W(t), h)$ is a martingale adapted to $\mathcal{F}_t$.
\end{enumerate}
\end{definition}
If $W$ is a an $H$-valued Wiener process with covariance operator
$Q$ with $\Tr Q < \infty$, then $W$ is a Gaussian process on $H$ and
$$ E(W(t)) = 0,\quad \text{Cov}\ (W(t)) = tQ, \quad t\geq 0.$$ Let
$H_0 = Q^{1/2}H.$ Then $H_0$ is a Hilbert space equipped with the
inner product $(\cdot, \cdot)_0$,
$$(u, v)_0 = \left(Q^{-1/2}u, Q^{-1/2}v\right),\ \forall u, v\in H_0,$$
where $Q^{-1/2}$ is the pseudo-inverse of $Q^{1/2}$. Since $Q$ is a
trace class operator, the imbedding of $H_0$ in $H$ is
Hilbert-Schmidt.

Let $L_Q$ denote the space of linear operators $S$ such that $S
Q^{1/2}$ is a Hilbert-Schmidt operator from $H$ to $H$. Define the
norm on the space $\mathrm{L}_Q$ by $|S|_{\mathrm{L}_Q}^2 =
\Tr(SQS^*)$.

\begin{definition}
Let $I=[a,b]$ be an interval in $\mathbb{R}^+.$ A mapping $g : I
\rightarrow \mathbb{R}^d$ is said to be c\`{a}dl\`{a}g if, for all
$t\in[a,b]$, $g$ has a left limit at $t$ and $g$ is right continuous
at $t$, i.e.,
\begin{itemize}
  \item [(i)] for all sequences $(t_n,n\in \mathbb{N})$ in $(a,b)$ with each $t_n<t$ and $lim_{n\to\infty}t_n = t$ we have that
  $lim_{n\to\infty}g(t_n)$ exists;
  \item [(ii)] for all sequences $(t_n,n\in \mathbb{N})$ in $(a,b)$ with each $t_n\geq t$ and $lim_{n\to\infty}t_n = t$ we have that
  $lim_{n\to\infty}g(t_n) = g(t);$
  \item [(iii)] for the end-points we stipulate that $g$ is right continuous at $a$ and has left limit at $b$.
\end{itemize}
\end{definition}
\begin{definition}
Let $(\Omega, \mathcal{F}, \mathcal{F}_t, P)$  be a filtered
probability space, and $E$ be a Banach space. A process
$(X_t)_{t\geq 0}$ with state space $(E, \mathcal{B}(E))$ is called a
L\'{e}vy process if
\begin{itemize}
\item [(i)] $(X_t)_{t\geq 0}$ is adapted to $(\mathcal{F}_t)_{t\geq
0}$,
\item [(ii)] $X_0 = 0$ a.s.,
\item [(iii)] $(X_t)_{t\geq 0}$ has increments of the past, i.e.
$X_t-X_s$ is independent of $\mathcal{F}_s$ if $0\leq s< t$,
\item [(iv)] $(X_t)_{t\geq 0}$ is stochastically continuous, i.e.
$\forall \e >0, \lim_{s\rightarrow t} P(|X_s - X_t| > \e) = 0$,
\item [(v)] $(X_t)_{t\geq 0}$ is c\`{a}dl\`{a}g,
\item [(vi)] $(X_t)_{t\geq 0}$ has stationary increments, i.e. $X_t-X_s$
has the same distribution as $X_{t-s}, 0\leq s < t$.
\end{itemize}
\end{definition}The jump of $X_t$ at $t\geq 0$ is given by $\triangle X_t = X_t -
X_{t-}$. Let $Z\in \mathcal{B}(\mathbb{R}^+ \times E)$. We define $$
N(t, Z)= N(t, Z, \omega)= \sum_{s: 0<s\leq t} \chi_{_Z} (\triangle
X_s).$$

In other words, $N(t, Z)$ is the number of jumps of size $\triangle
X_s\in Z$ which occur before or at time $t$. $N(t, Z)$ is called the
\emph{Poisson random measure} (or \emph{jump measure}) of
$(X_t)_{t\geq 0}$. The differential form of this measure is written
as $N(dt, dz)(\omega)$.

We call $\tilde{N}(dt, dz) = N(dt, dz) - dt \lambda(dz)$ a
\emph{compensated Poisson random measure (cPrm)}, where $dt
\lambda(dz)$ is known as \emph{compensator} of the L\'{e}vy process
$(X_t)_{t\geq 0}$. Here $dt$ denotes the Lesbegue measure on
$\mathcal{B}(\mathbb{R}^{+})$, and $\lambda(dz)$ is a
$\sigma$-finite measure on $(Z, \mathcal{B}(Z))$.

\begin{lemma}
If $X = (X_t)_{t\geq0}$ is a L\'{e}vy process, then $X_t$ is
infinitely divisible for each $t\geq0$
\end{lemma}
\noindent For proof see Proposition 1.3.1 of \cite{Ap}.

\begin{lemma}
If $X = (X_t)_{t\geq0}$ is a L\'{e}vy process, then
$$ \phi_{X_t}(u) = e^{t\eta(u)},$$
for each $u\in \mathbb{R}^d$, $t\geq0$, where $\eta$ is the L\'{e}vy
symbol of $X_1$.
\end{lemma}
\noindent For proof see Theorem 1.3.3 of \cite{Ap}.

\begin{lemma}
If $X = (X_t)_{t\geq0}$ is stochastically continuous, then the map
$t \rightarrow \phi_{X_t}(u)$ is continuous for each $u\in
\mathbb{R}^d$.
\end{lemma}
\noindent For proof see Lemma 1.3.2 of \cite{Ap}.

\begin{lemma}
If $X = (X_t)_ {t\geq0}$ is a stochastic process and there exists a
sequence of L\'{e}vy processes $(X_n, n\in\mathbb{N})$ with each
$X_n = (X_{n_t}, t\geq0)$ such that $X_{n_t}$ converges in
probability to $X_t$ for each $t\geq0$ and $$\lim_{n \to \infty}
\limsup_{t\rightarrow0} P(|X_{n_t}-X_t|>a) = 0,$$ for all $a>0$,
then $X$ is a L\'{e}vy process.
\end{lemma}
\noindent For proof see Theorem 1.3.7 of \cite{Ap}.

\begin{example}(Brownian motion)
A (standard) Brownian motion in $\mathbb{R}^d$ is a L\'{e}vy process
$B = (B_t)_{t\geq0}$ for which
\begin{enumerate}
\item[\textbf{(B1)}] $B_t$ $\sim N(0,tI)$ for each $t\geq0$,
\item[\textbf{(B2)}] B has continuous sample paths.
\end{enumerate}
It follows immediately from $(B1)$ that if B is a standard Brownian
motion then its characterestic function is given by
$$\phi_{B_t}(u) = exp\left(-\frac{1}{2}t|u|^2\right)$$
for each $u\in \mathbb{R}^d, t\geq0.$
\end{example}

\begin{example}(The Poisson Process) The Poisson process of intensity $\lambda>0$ is a L\'{e}vy
process N taking values in $\mathbb{N}\cup{\{0\}}$ wherein each
$N(t) \sim \pi(\lambda t)$, so that we have
$$P(N(t) = n) = \frac{(\lambda t)^n}{n!}e^{-\lambda t}$$
for each n =0, 1, 2, ....

  The \emph{compensated Poisson Process} $\tilde{N} = (\tilde{N}(t), t\geq0)$ where each $\tilde{N}(t) = N(t)-\lambda t.$
Note that $\mathbb{E}(\tilde{N}(t)) = 0.$ and
$\mathbb{E}(\tilde{N}(t)^2) = \lambda t$ for each $t\geq0$.
\end{example}

\begin{lemma}
Let $Z$ be bounded below, then $N(t,Z)<\infty$ (a.s) for all
$t\geq0.$
\end{lemma}
\noindent For proof see Lemma 2.3.4 of \cite{Ap}.

\begin{lemma}
\sni
\begin{enumerate}
\item [(i)] If Z is bounded below, then $(N(t,Z),t\geq0)$ is a Poisson process with intensity $\lambda(Z)$.
\item [(ii)] If $Z_1,\ldots,Z_m \in\mathcal{B}(\mathbb{R}^d-\{0\})$ are disjoint and bounded below and if $s_1,\ldots,s_m$
$\in\mathbb{R}^+$ are distinct, then the random variables
$N(s_1,Z_1),\ldots, N(s_m,A_m)$ are independent.
\end{enumerate}
\end{lemma}
\noindent For proof see Theorem 2.3.5 of \cite{Ap}.

\begin{lemma}
Every L\'{e}vy process is a semimartingale.
\end{lemma}
\noindent For proof see Proposition 2.7.1 of \cite{Ap}.

\begin{definition}(Poisson integration)
Let N be the Poisson measure associated to a L\'{e}vy process $X =
(X_t)_{t\geq0}.$
  Let $g$ be a measurable function from $\mathbb{R}^d$ into $\mathbb{R}^d$ and let $Z$ be bounded below; then for each
$t>0$, $\omega\in\Omega,$ we may define the \emph{Poisson integral}
of $g$ as random finite sum by
$$\int_Zg(z)N(t,dz)(\omega) = \sum_{z \in Z}{ g(z)N(t,\{z\})(\omega)}$$
Note that each $\int_Zg(z)N(t,dz)$ is an $\mathbb{R}^d$ valued
random variable and gives rise to a c\`{a}dl\`{a}g stochastic
process as we vary t. \sni Now, since $N(t,\{x\}) \neq 0
\Leftrightarrow \Delta X_u = x$ for at least one $0\leq u\leq t,$ we
have
\begin{equation}\label{1}
    \int_Zg(z)N(t,dz) = \sum_{0 \leq u \leq t} {g(\Delta X_u)\chi_Z(\Delta X_u)}.
\end{equation}

\end{definition}

\begin{lemma}(The L\'{e}vy-It\^{o} decomposition)
If $X = (X_t)_{t\geq0}$ is a L\'{e}vy process, then there exists
$b\in\mathbb{R}^d$, a Brownian motion $B_A$ with covariance matrix A
and an independent Poisson random measure N on $\mathbb{R}^+ \times
(\mathbb{R}^d - \{0\})$ such that, for each $t\geq0$,
$$X_t = bt+B_A(t)+\int_{|z|<1}z\tilde{N}(t,dz)+\int_{|z|\geq1}zN(t,dz).$$
\end{lemma}
\noindent For proof see Theorem 2.4.16 of \cite{Ap}.

\begin{definition}
Let $E$ and $F$ be separable Banach spaces. Let $F_t:=
\mathcal{B}(\mathbb{R}^+ \times E)\otimes\mathcal{F}_t$ be the
product $\sigma$-algebra generated by the semi-ring $
\mathcal{B}(\mathbb{R}^+ \times E)\times\mathcal{F}_t$ of the
product sets $Z\times F,\quad Z\in\mathcal{B}(\mathbb{R}^+ \times
E),\quad F\in \mathcal{F}_t$ ( where $\mathcal{F}_t$ is the
filtration of the additive process $(X_t)_{t\geq0})$. Let $T>0$, and
\begin{align*}
\mathbb{H}(Z) = & \Big\{g : \mathbb{R}^+ \times Z \times \Omega
\rightarrow F, \quad such \quad that \quad g \quad is \quad
F_T/\mathcal{B}(F)\nonumber\\ & measurable\quad and \quad
g(t,z,\omega)\quad is\quad \mathcal{F}_t - adapted\quad \forall z\in
Z, &\forall t\in (0,T]\Big\}
\end{align*}
Let $p\geq1$,
$$\mathbb{H}^{p}_\lambda([0,T]\times Z;F) = \left\{g\in \mathbb{H}(Z) : \int_0^T\int_Z\mathbb{E}[\|g(t,z,\omega)\|^p_F]\lambda(dz)dt < \infty \right\}$$
\end{definition}

Let $H$ be a vector with components $(H^1,H^2,\ldots,H^d)$ taking
values in $\mathbb{H}^2_\lambda([0,T]\times Z;E)$; then we may
construct an $\mathbb{R}^d$-valued process $A = (A(t),t \geq 0)$
with components $(A^1,A^2,\ldots,A^d)$ where each
$$ A^i(T) = \int_0^T\int_{|z|\leq1} H^i(t,z)\tilde{N}(dt,dz).$$
The construction of $A$ extends to the case where $H$ is no longer
lies in $\mathbb{H}^2_\lambda([0,T]\times Z;E)$ but satisfies
$$P \left(\int_0^T\int_E|H(t,z)|\lambda(dz)dt<\infty \right) = 1.$$
In this case $A$ is still a local martingale. It is an
$L^1$-martingale if $$\int_0^T\int_E
\mathbb{E}(|H(t,z)|)\lambda(dz)dt < \infty.$$ Let us introduce the
compound Poisson process $P = (P_t, t\geq0)$, where each $P(t) =
\int_ZzN(t,dz).$ Let K be a predictable mapping; then, generalizing
equation (1.1), we define
\begin{equation}\label{poi}
    \int_0^T\int_ZK(t,z)N(dt,dz) = \sum_{0\leq u\leq t}K(u,\Delta P_u)\chi_Z(\Delta P_u)
\end{equation}
as a random finite sum.

  In particular, if $H$ satisfies the square-integrability ( or integrability ) condition given above we may then define, for each $1\leq i\leq d$,
$$\int_0^T\int_ZH^i(t,z)\tilde{N}(dt,dz) = \int_0^T\int_ZH^i(t,z)N(dt,dz) - \int_0^T\int_ZH^i(t,z)\lambda(dz)dt.$$

\begin{definition}
An $\mathbb{R}^d$-valued stochastic process $Y=(Y_t)_{t\geq0}$ is a
\emph{L\'{e}vy-type stochastic integral} if it can be written in the
following form, for each $1\leq i\leq d$, $t\geq0$, $1\leq i\leq d$,
$1\leq j\leq m$, $t\geq0$, we have $|G^i|^{1/2}$, $F^i_j \in
\mathrm{L}^2[0,T],H^i \in \mathbb{H}^2_\lambda([0,T]\times Z;E)$ and
$K$ is predictable:
\begin{align}\label{lev}
 Y^i(t) = &Y^i(0) + \int_0^tG^i(s)ds + \int_0^tF^i_j(s)dB^j(s) + \int_0^t\int_{|z|<1}H^i(s,z)\tilde{N}(ds,dz)\nonumber\\ & + \int_0^t\int_{|z|\geq1}K^i(s,z)N(ds,dz)
\end{align}
Here $B$ is an $m-$dimensional standard Brownian motion and $N$ is
an independent Poisson random measure on $\mathbb{R}^+ \times
(\mathbb{R}^d - \{0\})$ with compensator $\tilde{N}$ and intensity
measure $\lambda$, which is a L\'{e}vy measure.
\end{definition}We often simplify complicated expressions by employing the notation of \emph{stochastic differentials} to represent L\'{e}vy-type stochastic integrals.
We then write \eqref{lev} as
$$dY(t) = G(t)dt + F(t)dB(t) + H(t,z)\tilde{N}(dt,dz) + K(t,z)N(dt,dz).$$
When we want particularly to emphasize the domain of integration
with respect to z, we will use the equivalent notation
$$dY(t) = G(t)dt + F(t)dB(t) + \int_{|z|<1}H(t,z)\tilde{N}(dt,dx) + \int_{|z|\geq1}K(t,z)N(dt,dz).$$
Clearly Y is a semi martingale.

Let Y be a general L\'{e}vy-type stochastic process with stochastic
differential
\begin{align}\label{ito}
dY^i(t) = &G^i(t)dt + F^i_j(t)dB^j(t) +
\int_{|z|<1}H^i(t,z)\tilde{N}(dt,dz) \nonumber\\ &+
\int_{|z|\geq1}K^i(t,z)N(dt,dz).
\end{align}
where, for each $1\leq i\leq d$, $1\leq j\leq m$, $t\geq0$,
$|G^i|^{1/2}$, $F^i_j\in\mathrm{L}^2[0,T]$ and
$H^i\in\mathbb{H}^2_\lambda([0,T]\times Z;E).$ Let
$$dY_c(t) = G^i(t)dt + F^i_j(t)dB^j(t),$$
and the discontinuous part of $Y$
$$dY_d(t) = \int_{|z|<1}H^i(t,z)\tilde{N}(dt,dz) + \int_{|z|\geq1}K^i(t,z)N(dt,dz),$$
so that for each $t\geq0$
$$Y(t) = Y(0) + Y_c(t) + Y_d(t).$$

\begin{assumption}
For all $t>0$,
$$\sup_{0\leq s\leq t} \sup_{0<|z|<1}|H(s,z)|<\infty\qquad a.s$$
\end{assumption}

\begin{lemma}(It\^{o}'s theorem 1)
If $Y = (Y_t)_{t\geq0}$ is a L\'{e}vy-type stochastic integral of
the form \eqref{ito}, then, for each $f \in C^2(\mathbb{R}^d),$
$t\geq0$, with probability 1 we have
\begin{align}
f(Y(t)) - &f(Y(0)) \nonumber\\ &=
\int_0^t\partial_if(Y(s-))dY_c^i(s) +
\frac{1}{2}\int_0^t\partial_i\partial_jf(Y(s-))d[Y_c^i,Y_c^j](s)
\nonumber\\ &+ \int_0^t\int_{|z|\geq1}\left[f(Y(s-) + K(s,z)) -
f(Y(s-))\right]N(ds,dz) \nonumber\\ &+
\int_0^t\int_{|z|<1}\left[f(Y(s-) + H(s,z)) -
f(Y(s-))\right]\tilde{N}(ds,dz)\nonumber\\ & +
\int_0^t\int_{|z|<1}\left[f(Y(s-) + H(s,z)) - f(Y(s-))
\right.\nonumber\\ &\left.-
H^i(s,z)\partial_if(Y(s-))\right]\lambda(dz)ds.
\end{align}
\end{lemma}
\noindent For proof see Theorem 4.4.7 of \cite{Ap}.

\begin{definition}
Let $M$ be a Brownian integral with the drift of the form
$$M^i(t)=\int_0^tF^i_j(s)dB^j(s) + \int_0^tG^i(s)ds,$$
where each $F^i_j , (G^i)^{1/2}\in\mathrm{L}^2[0,T]$ for all
$t\geq0$, $1\leq i\leq d,$ $1\leq j\leq m.$

  For each $1\leq i\leq j,$ the \emph{quadratic variation process}, denoted as $([M^i,M^j](t), t\geq0),$ is defined by
$$[M^i,M^j](t) = \sum_{k=1}^m\int_0^tF^i_k(s)F^j_k(s)ds.$$
\end{definition}

\begin{lemma}(Burkholder's Inequality)
Let $M = (M(t), t\geq0)$ be a (real-valued) Brownian integral of the
form
$$M(t) = \int_0^tF^j(s)dB_j(s),$$
where each $F^j \in \mathrm{L}^2[0,T], 1\leq j\leq d, t\geq0.$ Let
$$[M,M](t) = \sum_{j=1}^m\int_0^tF_j(s)^2ds,$$
for each $t\geq0$. Then $M$ is a square-integrable martingale. Let
$\mathbb{E}([M,M](t)^{p/2}) < \infty,$ then for any $p\geq2$ there
exists a $C(p)>0$ such that, for each $t\geq0,$
$$\mathbb{E}\left(|M(t)|^p\right) \leq C(p)\mathbb{E}\left([M,M](t)^{p/2}\right).$$
\end{lemma}
\noindent For proof see Theorem 4.4.21 of \cite{Ap}.

\begin{lemma}(Burkholder-Davis-Gundy Inequality)
For every $p \geq 1,$ there is a constant $C_p \in (0,\infty)$ such
that for any real-valued square integrable c\`{a}dl\`{a}g martingale
$M$ with $M_0 = 0$, and for any $T \geq 0$,
$$C_p^{-1}\mathbb{E}[M,M]_T^{p/2} \leq \mathbb{E}\sup_{0\leq t\leq T}|M_t|^p \leq C_p\mathbb{E}[M,M]_T^{p/2}.$$
\end{lemma}
\noindent For proof see Theorem 3.50 of \cite{PZ}

\begin{definition}
Let $g : \mathbb{R}^+ \times Z \times \Omega \rightarrow F$ be
given. A sequence $\{g_n\}_{n\in\mathbb{N}}$ of $F_T/\mathcal{B}(F)$
measurable functions is $\mathrm{L}^p$- approximating $g$ on $(0,T]
\times Z \times \Omega$ w.r.t. $\lambda \otimes \mathrm{P}$, if
$g_n$ is $\lambda \otimes \mathrm{P}$-a.s. converging to $g$, when
$n \rightarrow \infty$, and
$$\lim_{n \rightarrow \infty}\int_0^T\int_Z\mathbb{E}\big[\|g_n(t,z,\omega) - g(t,z,\omega)\|^p\big]\d \lambda = 0,$$
i.e., $\|g_n-g\|$ converges to zero in $\mathrm{L}^p((0,T] \times Z
\times \Omega, \lambda \otimes \mathrm{P}),$ when $n \rightarrow
\infty$.
\end{definition}

\begin{definition}
Let $p\geq1, T>0$. We say that $g$ is strong $p$-integrable on
$(0,T]\times Z, Z\in\mathcal{B}(E),$ if there exists a sequence
$\{g_n\}_{n\in N} \in \Sigma(E)$ of simple functions, such that
$g_n$ is $ \mathrm{L}^p$-approximating $g$ on $(0,T]\times Z\times
\Omega$ w.r.t. $\lambda\otimes P,$ and for any such sequence the
limit of the natural integrals of $g_n$ w.r.t. $\tilde{N}(dt,dz)$
exists in $\mathrm{L}^p(\Omega, \mathcal{F},P)$ for
$n\rightarrow\infty,$ i.e.,
\begin{align}\label{stro}
\int_0^T\int_Zg(t,z,\omega)\tilde{N}(dt,dz)(\omega) =
\lim_{n\rightarrow\infty}^p\int_0^T\int_Zg_n(t,z,\omega)\tilde{N}(dt,dz)(\omega)
\end{align}
exists. Moreover, the limit \eqref{stro} does not depend on the
sequence $\{g_n\}_{n\in\mathbb{N}}\in\Sigma(E),$ which is
$\mathrm{L}^p$-approximating $g$ on $(0,T]\times Z\times \Omega$
w.r.t. $\lambda\otimes P$. We call the limit in \eqref{stro} the
strong $p$-integral of $g$ w.r.t. $\tilde{N}(dt,dz)$ on $(0,T]\times
Z.$
\end{definition}

\begin{lemma}
Let $p\geq1.$ Let g be strong $p$ -integrable on $(0,T]\times Z,
Z\in\mathcal{B}(E).$ Then the strong p-integral
$\int_0^t\int_Zg(s,z,\omega)\tilde{N}(ds,dz)(\omega),$ $t\in[0,T],$
is an $\mathcal{F}_t$-martingale with mean zero.
\end{lemma}
\noindent For proof see Theorem 4.10 of \cite{Ru}.

\begin{lemma}\label{le1}
Let $f\in \mathrm{L}^1([0,T],E)$; then $f$ is strong 1-integrable
w.r.t $\tilde{N}(dt,dz)$ on $(0,T]\times Z$, for any $0<t \leq T$,
$Z\in \mathcal{B}(E)$. Moreover
$$\mathbb{E}\left[\left\|\int_0^t\int_Z
f(s,z,\omega)\tilde{N}(ds,dz)\right\|\right]\leq 2 \int_0^t\int_Z
\mathbb{E}\left[\|f(s,z,\omega)\|\right]\lambda(dz)ds$$
\end{lemma}
\noindent For proof see Theorem 4.12 of \cite{Ru}.

\begin{lemma}
Suppose that $(F,\mathcal{B}(F)) = (H,\mathcal{B}(H))$ is a
separable Hilbert space. Let $g\in \mathbb{H}^2_\lambda(E),$ then
$g$ is strong $2$-integrable w.r.t. $\tilde{N}(dt,dz)$ on
$(0,T]\times Z,$
 for any $0 < t \leq T,$ $Z\in\mathcal{B}(E)$. Moreover
\begin{align}
\mathbb{E}\left[\Big\|\int_0^t\int_Zg(s,z,\omega)\tilde{N}(ds,dz)(\omega)\Big\|^2\right]
=
\int_0^t\int_Z\mathbb{E}\left[\|g(s,z,\omega\|^2\right]\lambda(ds)dz
\end{align}
\end{lemma}
\noindent For proof see Theorem 4.14 of \cite{Ru}.

\section{The Stochastic GOY Model of Turbulence}
\setcounter{equation}{0} The GOY model
(Gledger-Ohkitani-Yamada)~\cite{OY} is a particular case of so
called ``Shell model" (see, Frisch \cite{Fu}). This model is the
Navier-Stokes equation written in the Fourier space where the
interaction between different modes is preserved between nearest
modes. To be precise, the GOY model describes a one-dimensional
cascade of energies among an infinite sequence of complex
velocities, $\{u_n(t)\}$, on a one dimensional sequence of wave
numbers
$$k_n = k_0 2^n,\quad k_0 > 0,\ n=1, 2, \ldots$$
where the discrete index $n$ is referred to as the ``shell index".
The equations of motion of the GOY model of turbulence have the form
\begin{align}\label{goy}
\frac{\d u_n}{\d t} + \nu k_n^2 u_n &+ i\big(a k_n
u\s_{n+1}u\s_{n+2} + b k_{n-1}u\s_{n-1}u\s_{n+1} + \nonumber\\ &
+ck_{n-2} u\s_{n-1}u\s_{n-2}\big) = f_n, \quad\text{for}\ n= 1, 2,
\ldots,
\end{align}
along with the boundary conditions
\begin{equation}\label{bc}
u_{-1} = u_0 = 0.
\end{equation}
Here $u\s_n$ denotes the complex conjugate of $u_n$, $\nu > 0$ is
the kinematic viscosity and $f_n$ is the Fourier component of the
forcing. $a, b$ and $c$ are real parameters such that energy
conservation condition $a + b + c =0$ holds (see Kadanoff, Lohse,
Wang, and Benzi\cite{Ka}; Ohkitani and Yamada\cite{OY}).

\subsection{Functional Setting}
Let $H$ be a real Hilbert space such that
\begin{equation*}
H :=\left\{u=(u_1, u_2, \ldots) \in \mathbb{C}^{\infty}:
\sum_{n=1}^{\infty} |u_n|^2 < \infty\right\}.
\end{equation*}
For every $u, v \in H$, the scalar product $(\cdot,\cdot)$ and norm
$|\cdot|$ are defined on $H$ as
$$(u, v)_H = Re\ \sum_{n=1}^{\infty} u_n v\s_n, \quad |u| = \left(\sum_{n=1}^{\infty} |u_n|^2\right)^{1/2}.$$
Let us now define the space
$$V :=\left\{u\in H: \sum_{n=1}^{\infty} k_n^2|u_n|^2 < \infty\right\},$$
which is a Hilbert space equipped with the norm
$$\|u\| = \left(\sum_{n=1}^{\infty} k_n^2|u_n|^2\right)^{1/2}.$$
The linear operator $A: D(A) \rightarrow H$ is a positive definite,
self adjoint linear operator defined by
\begin{equation}\label{A}
Au=((Au)_1, (Au)_2, \ldots), \ \text{where}\ (Au)_n = k_n^2 u_n,
\quad\forall u\in D(A).
\end{equation}
The domain of $A$, $D(A) \subset H$, is a Hilbert space equipped
with the norm
$$\|u\|_{D(A)} = |Au| = \left(\sum_{n=1}^{\infty} k_n^4|u_n|^2\right)^{1/2}, \quad\forall u\in D(A).$$
Since the operator $A$ is positive definite, we can define the power
$A^{1/2}$ ,
$$A^{1/2}u = (k_1 u_1, k_2 u_2, \ldots), \quad\forall u=(u_1, u_2, \ldots).$$
Furthermore, we define the space
$$D(A^{1/2}) = \left\{u=(u_1, u_2, \ldots): \sum_{n=1}^{\infty} k_n^2 |u_n|^2 < \infty\right\}$$
which is a Hilbert space equipped with the scalar product
$$ (u, v)_{D(A^{1/2})} = (A^{1/2}u, A^{1/2}v), \quad\forall u, v\in D(A^{1/2}),$$
and the norm $$\|u\|_{D(A^{1/2})} = \left(\sum_{n=1}^{\infty} k_n^2
|u_n|^2\right)^{1/2}.$$ Note that $V = D(A^{1/2})$. We consider
$V^{\prime} = D(A^{-1/2})$ as the dual space of $V$. Then the
following inclusion holds
$$V\subset H = H^{\prime}\subset V^{\prime}.$$
We will now introduce the sequence spaces analogue to Sobolev
functional spaces. For $1\leq p <\infty$ and $s\in\mathbb{R}$
$$\mathrm{W}^{s, p} :=\left\{u = (u_1, u_2, \ldots): \|A^{s/2}u\|_p = \left(\sum_{n=1}^{\infty}(k_n^s|u_n|)^p\right)^{1/p} < \infty\right\},$$
and for $p=\infty$
$$\mathrm{W}^{s, \infty} :=\left\{u = (u_1, u_2, \ldots): \|A^{s/2}u\|_{\infty} = \sup_{1\leq n<\infty} (k_n^s|u_n|) < \infty\right\},$$
where for $u\in\mathrm{W}^{s, p}$ the norm is defined as
$$\|u\|_{\mathrm{W}^{s, p}} = \|A^{s/2}u\|_p.$$
Here $\|\cdot\|$ denotes the usual norm in the $l^p$ sequence space.
It is clear from the above definitions that $W^{1, 2} = V =
D(A^{1/2})$.

\begin{remark}
For the shell model we can reasonably assume that the complex
velocities $u_n$ are such that $|u_n| <1$ for almost all $n$. Then
\begin{align*}
\| u \|_{l^4}^4 = \sum_{n=1}^{\infty} |u_n|^4 \leq
\left(\sum_{n=1}^{\infty} |u_n|^2\right)^2 = |u|^4,
\end{align*}
which leads to $H\subset l^4$.
\end{remark}

We now state a Lemma which is useful in this work. We omit the proof
since it is quite simple.
\begin{lemma}\label{La}
For any smooth function $u\in H$, the following holds:
\begin{align}
  \| u \|_{l^4}^4 \leq C |u|^2 \ \|u\|^2.\label{L4}
\end{align}
\end{lemma}

\subsection{Properties of the linear and nonlinear operators}
We define the bilinear operator $B(\cdot, \cdot): V \times
H\rightarrow H$ as
$$B(u, v) = (B_1(u,v), B_2(u, v), \ldots),$$ where
\begin{align*}
B_n(u, v)= ik_n\left(\frac{1}{4}u\s_{n+1} v\s_{n-1} -
\frac{1}{2}(u\s_{n+1} v\s_{n+2} + u\s_{n+2} v\s_{n+1})
 + \frac{1}{8}u\s_{n-1} v\s_{n-2}\right).
\end{align*}
In other words, if $\{e_n\}_{n=1}^{\infty}$ be a orthonormal basis
of $H$, i.e. all the entries of $e_n$ are zero except at the place
$n$ it is equal to $1$, then
\begin{align}\label{B}
B(u, v)= i\sum_{n=1}^{\infty}k_n\left(\frac{1}{4}u\s_{n+1} v\s_{n-1}
- \frac{1}{2}(u\s_{n+1} v\s_{n+2} + u\s_{n+2} v\s_{n+1})
 + \frac{1}{8}u\s_{n-1} v\s_{n-2}\right)e_n.
\end{align}

 The following lemma says that $B(u, v)$ makes
sense as an element of $H$, whenever $u\in V$ and $v\in H$ or $u\in
H$ and $v\in V$. It also says that $B(u, v)$ makes sense as an
element of $V^{\prime}$. Here we state the following lemma which has
been proved in Constantin, Levant and Titi \cite{CLT} for the Sabra
shell model, but one can also prove the similar estimates for the
GOY model (see Barbato, Barsanti, Bessaih, and Flandoli\cite{Ba}).
\begin{lemma}\label{Bprop1}
(i) There exist constants $C_1 >0, C_2 >0$,
\begin{equation}
|B(u, v)| \leq C_1 \|u\| |v|, \quad\forall u\in V, v\in H,
\end{equation}
and
\begin{equation}
|B(u, v)| \leq C_2 |u| \|v\|, \quad\forall u\in H, v\in V.
\end{equation}
(ii) $B: H\times H\rightarrow V^{\prime}$ is a bounded bilinear
operator and for a constant $C_3 > 0$
\begin{equation}
\|B(u, v)\|_{V^{\prime}} \leq C_3 |u| |v|, \quad\forall u, v\in H.
\end{equation}
(iii) $B: H\times D(A)\rightarrow V$ is a bounded bilinear operator
and for a constant $C_4 > 0$
\begin{equation}
\|B(u, v)\|_{V} \leq C_4 |u| |Av|, \quad\forall u\in H, v\in D(A).
\end{equation}
(iv) For every $u\in V$ and $v\in H$
\begin{equation}\label{1}
(B(u, v), v) = 0.
\end{equation}
\end{lemma}

We now present one more important property of the nonlinear operator
$B$ in the following lemma which will play important role in the
later part of this section. The proof is straightforward and uses
the bilinearity property of $B$.
\begin{lemma}\label{Bprop2}
If $w=u-v$, then
$$B(u, u)-B(v, v) = B(v, w) + B(w, v) + B(w, w).$$
\end{lemma}

With above functional setting and following the classical treatment
of the Navier-Stokes equation, one can write the stochastic GOY
model of turbulence \eqref{goy} with the L\'{e}vy forcing as the
following,
\begin{align}\label{goy1}
&\d u + \big[\nu Au + B(u, u)\big] \d t = f(t) \d t + \sqrt{\varepsilon}\sigma(t, u) \d W(t) + \varepsilon\int_Zg(u,z)\tilde{N}(dt,dz)\\
& u(0) = u_0,
\end{align}
where $u \in H$, the operators $A$ and $B$ are defined through
\eqref{A} and \eqref{B} respectively,  $f=(f_1, f_2, \ldots),
\sigma(t, u)=(\sigma_1(t, u_1), \sigma_2(t, u_2), \ldots)$. Here
$(W(t)_{t\geq 0})$ is a $H$-valued Wiener process with trace class
covariance, and the space $L_Q$ has been defined in section 1. Here
$g(u,z)$ is a measurable mapping from $H \times Z$ into H and let $
\mathcal{D}([0,T],H) $ be the space of all c\`{a}dl\`{a}g paths from
$[0,T]$ into H.

Assume that $\sigma$ and $g$ satisfy the following hypotheses of
joint continuity, Lipschitz condition and linear growth:
\begin{hypothesis}\label{hyp}
The main hypothesis is the following,
\begin{itemize}
\item[H.1.]  The function $\sigma \in C([0, T] \times V; L_Q(H_0; H))$, and $g\in \mathbb{H}^2_{\lambda}([0, T] \times Z; H)$.

\item[H.2.]  For all $t \in (0, T)$, there exists a positive constant $K$ such that for all $u \in H$,
$$|\sigma(t, u)|^2_{L_Q} + \int_{Z} |g(u, z)|^2_{H}\lambda(dz) \leq K(1 +|u|^2).$$

\item[H.3.]  For all $t \in (0, T)$,  there exists a positive constant $L$ such that for all $u, v \in H$,
$$|\sigma(t, u) - \sigma(t, v)|^2_{L_Q} + \int_{Z} |g(u, z)-g(v, z)|^2_{H}\lambda(dz)\leq L|u - v|^2.$$
\end{itemize}
\end{hypothesis}

The following lemma shows that sum of the linear and nonlinear
operator is locally monotone in the $l^4$-ball.
\begin{lemma}\label{Mon}
For a given $r > 0$, let us denote by $\mathbb{B}_r$ the closed
$l^4$-ball in $V$:
$$\mathbb{B}_r = \Big\{v\in V; \|v\|_{l^4} \leq r\Big\}.$$
Define the nonlinear operator F on $V$ by $F(u):=-\nu Au - B(u, u)$.
Then for $0 < \varepsilon < \frac{\nu}{2L}$, where $L$ is the
positive constant that appears in the condition (H.3), the pair $(F,
\sqrt{\varepsilon}\sigma + \varepsilon\int_Zg(.,z)\lambda(dz))$ is
monotone in $\mathbb{B}_r$, i.e. for any $u\in V$ and $v\in
\mathbb{B}_r$
\begin{equation}\label{monotone}
(F(u) - F(v), w)  - \frac{r^4}{\nu^3} |w|^2 +
\varepsilon\big[|\sigma(t,u)-\sigma(t,v)|^2_{L_Q} +
\int_Z|g(u,z)-g(v,z)|^2\lambda(dz)\big]\leq0
\end{equation}
where $w = u - v$.
\end{lemma}
\begin{proof}
First note that, $$\nu(Aw, w) = \nu\|w\|^2.$$ Next using the Lemma
\ref{Bprop2} and equation\eqref{1} from Lemma \ref{Bprop1}, we have
$$(B(u, u) - B(v, v), w) = (B(v, w) + B(w, v) + B(w, w), w) = (B(w, v), w).$$
Now using the definition of the operator $B$ and equation \eqref{L4}
from Lemma \ref{La}, we get for $C >0$,
\begin{align*}
\big|(B(w, v), w)\big| &= \big|\sum_{n=1}^{\infty} ik_n\big[\frac{1}{4}v\s_{n-1}w\s_{n+1} w\s_{n} - \frac{1}{2}(w\s_{n+1} v\s_{n+2} + w\s_{n+2} v\s_{n+1})w\s_n + \\
&\quad + \frac{1}{8}w\s_{n-1} v\s_{n-2}w\s_n\big]\big|\\
&\leq C\|v\|_{l^4} \|w\|_{l^4} \|w\|\\
&\leq \|v\|_{l^4} |w|^{1/2} \|w\|^{3/2}\\
&\leq \frac{\nu}{2}\|w\|^2 + \frac{27}{32\nu^3} |w|^2 \|v\|_{l^4}^4.
\end{align*}
Since $v\in\mathbb{B}_r$, the above relation yields
$$- (B(w, v), w) \leq \frac{\nu}{2}\|w\|^2 + \frac{r^4}{\nu^3} |w|^2.$$
Hence by the definition of the operator $F$,
\begin{equation}\label{monotone2}
(F(u) - F(v), w) \leq -\frac{\nu}{2}\|w\|^2 + \frac{r^4}{\nu^3}
|w|^2.
\end{equation}
 We have
$$(F(u) - F(v), w) + \frac{\nu}{2}\|w\|^2 - \frac{r^4}{\nu^3} |w|^2\leq0.$$
But $V\subset H$ $\Rightarrow
\frac{\nu}{2}|w|^2\leq\frac{\nu}{2}\|w\|^2$. We get,
$$(F(u) - F(v), w) + \frac{\nu}{2}|w|^2 - \frac{r^4}{\nu^3} |w|^2\leq0.$$
Using condition (H.3) one can deduce that,
$$(F(u) - F(v), w)  - \frac{r^4}{\nu^3} |w|^2 + \frac{\nu}{2L}[|\sigma(t,u)-\sigma(t,v)|^2_{L_Q} + \int_Z|g(u,z)-g(v,z)|^2\lambda(dz)]\leq0 $$
Now choose $0 < \varepsilon < \frac{\nu}{2L}$ so that we get,
$$(F(u) - F(v), w)  - \frac{r^4}{\nu^3} |w|^2 + \varepsilon\big[|\sigma(t,u)-\sigma(t,v)|^2_{L_Q} + \int_Z|g(u,z)-g(v,z)|^2\lambda(dz)\big]\leq0$$
\end{proof}

\section{Energy Estimates and Existence Result}
\setcounter{equation}{0} Let $H_n := \text{span}\ \{e_1, e_2,
\cdots, e_n\}$ where $\{e_j\}$ is any fixed orthonormal basis in $H$
with each $e_j \in D(A)$. Let $P_n$ denote the orthogonal projection
of $H$ to $H_n$. Define $u^n = P_n u$, not to cause any confusion in
notation with earlier $u_n$. Let $W_n = P_nW$ . Let $\sigma_n =
P_n\sigma $ and
$\int_Zg^n(\u(t-),z)\tilde{N}(dt,dz)=P_n\int_Zg(u(t-),z)\tilde{N}(dt,dz)$,
where $g^n=P_ng$. Define $\u$ as the solution of the following
stochastic differential equation in the variational form such that
for each $v \in H_n$,
\begin{align}\label{variational}
\d (\u(t) , v)& = (F(\u(t)), v)\d t + (f(t), v)\d t +
\sqrt{\varepsilon}(\sigma_n(t, \u(t)) \d W_n(t), v) \nonumber\\ & +
\varepsilon\int_Z\big(g^n(\u(t-),z),v\big)\tilde{N}(dt,dz),
\end{align}
with $\u(0) = P_n u(0) $.

\begin{theorem}\label{energy}
Under the above mathematical setting let $f$ be in $\mathrm{L}^2(
[0, T], H)$,  $u (0)$ be $\mathcal{F}_0$ measurable, $\sigma \in
C([0, T] \times V; L_Q(H_0; H))$,  $g\in \mathbb{H}^2_{\lambda}([0,
T] \times Z; H)$ and $\mathbb{E}|u (0)|^2 < \infty$. Let $\u$ denote
the unique strong solution of the stochastic differential equation
\eqref{variational} in $\mathcal{D}([0, T], H_n)$. Then with $K$ as
in condition (H.2), the following estimates hold:

For all $\varepsilon$, and $0 \leq t \leq T$,
\begin{align}
&\mathbb{E}| \u(t)|^2 + \nu\int_0^t \mathbb{E}\| \u(s)\|^2 \d s
\nonumber\\ &\quad \leq \left(1 + \varepsilon KTe^{\varepsilon
KT}\right)\left(\mathbb{E}|u(0)|^2 +
\frac{1}{\nu}\int_0^t\|f(s)\|_{V^\prime}^2ds + \varepsilon KT\right)
,\label{energy1}
\end{align}
and for all $\varepsilon >0$,
\begin{align}
\mathbb{E}\left[\sup_{0\leq t\leq T} |\u(t)|^2 + 2\nu\int_0^T
\|\u(t)\|^2 \d t\right] \leq C\left(\mathbb{E}|u(0)|^2, \int_0^T
\|f(t)\|^2_{V^{\prime}} \d t, \nu, T\right)\label{energy2}.
\end{align}
\end{theorem}

\begin{proof}
Applying It\^{o}'s lemma to the function $|\u(t)|^2$ and using the
properties of the operators $A$ and $B$, we notice that,
\begin{align*}
&\d |\u(t)|^2 + 2\nu\|\u(t)\|^2 \d t \\
&\quad = 2(f(t),
\u(t))\d t + \varepsilon \Tr (\sigma_n(t,\u(t)) Q \sigma_n(t,\u(t)))\d t\\
&\quad\quad + 2\sqrt{\varepsilon} (\sigma_n(t, \u(t)), \u(t)) \d W_n(t)  + \varepsilon\int_Z|g^n(\u(s-),z)|^2N(ds,dz) \\
&\quad\quad +
2\varepsilon\int_Z\big(\u(s-),g^n(\u(s-),z)\big)\tilde{N}(ds,dz)
\end{align*}
Using the inequality $$ 2ab \leq \delta a^{2} + \frac{1}{\delta}
b^{2}$$ on $2 (f(t), \u(t))$ , we obtain
\begin{align*}
\d |\u(t)|^2 &+ 2\nu\|\u(t)\|^2 \d t \leq (\nu\|\u(t)\|^2 +
\frac{1}{\nu}\|f(t)\|^2_{V^{\prime}}) \d t \\ &+
\varepsilon|\sigma_n(t,\u(t))|^2\d t +
\varepsilon\int_Z|g^n(\u(s-),z)|^2N(ds,dz) \\ &+
2\sqrt{\varepsilon}(\sigma_n(t,\u(t)),\u(t))\d W_n(t) \\ &+
2\varepsilon\int_Z\big(\u(s-),g^n(\u(s-),z)\big)\tilde{N}(ds,dz).
\end{align*}
Define $$\tau_N = \inf\left\{t: |\u(t)|^2 + \int_0^t\|\u(s)\|^2 \d s
> N\right\}.$$ Then integrating one can deduce
\begin{align}\label{555}
|\u(\t)|^2 &+\nu\int_0^{\t}\|\u(s)\|^2 \d s \nonumber\\ &\quad\leq
|u(0)|^2 +
\frac{1}{\nu}\int_0^{\t}\|f(s)\|^2_{V^{\prime}} \d s + \int_0^{\t}\varepsilon|\sigma_n(s,\u(s))|^2\d s \nonumber\\
&\quad +\int_0^{\t}\varepsilon\int_Z|g^n(\u(s-),z)|^2N(ds,dz) \nonumber\\ &\quad+ 2\sqrt{\varepsilon}\int_0^{\t} (\sigma_n(s, \u(s)), \u(s)) \d W_n(s) \nonumber\\
&\quad+
2\int_0^{\t}\varepsilon\int_Z\big(\u(s-),g^n(\u(s-),z)\big)\tilde{N}(ds,dz).
\end{align}
Hence we can write this as
\begin{align}\label{505}
|\u(\t)|^2 &+\nu\int_0^{\t}\|\u(s)\|^2 \d s \nonumber\\
&\quad\leq |u(0)|^2 +\frac{1}{\nu}\int_0^{\t}\|f(s)\|^2_{V^{\prime}} \d s + \int_0^{\t}\varepsilon|\sigma_n(s,\u(s))|^2\d s \nonumber\\
&\quad +\int_0^{\t}\varepsilon\int_Z|g^n(\u(s-),z)|^2\lambda(dz)ds\nonumber\\
&\quad+ 2\sqrt{\varepsilon}\int_0^{\t} (\sigma_n(s, \u(s)), \u(s)) \d W_n(s) \nonumber\\
&\quad+\int_0^{\t}\e\int_Z|g^n(\u(s-),z)|^2\tilde{N}(ds,dz)\nonumber\\
&\quad+2\int_0^{\t}\varepsilon\int_Z\big(\u(s-),g^n(\u(s-),z)\big)\tilde{N}(ds,dz).
\end{align}
Using H\"{o}lder's inequality, one can note that if $g^n$ is strong
2-integrable w.r.t $\tilde{N}(dt,dz)$, $|g^n|^2$ is strong
1-integrable w.r.t $\tilde{N}(dt,dz)$. Hence taking expectation on
both sides of \eqref{505}, using the lemma \ref{le1} on $|g^n|^2$,
and using the fact that the stochastic integrals appeared in
$$2\sqrt{\varepsilon}\int_0^{\t} (\sigma_n(s, \u(s)), \u(s)) \d W_n(s)$$ and
$$2\int_0^{\t}\varepsilon\int_Z\big(\u(s-),g^n(\u(s-),z)\big)\tilde{N}(ds,dz)$$ are martingales, and having zero averages, we get
\begin{align*}
&\mathbb{E}\left[|\u(\t)|^2 +\nu \int_0^{\t}\|\u(s)\|^2 \d s\right] \\
&\leq \mathbb{E}|u(0)|^2
+\frac{1}{\nu}\int_0^{\t}\|f(s)\|^2_{V^{\prime}} \d s  \\ &\quad +
\int_0^{\t}\mathbb{E}\Big[\varepsilon|\sigma_n(s,\u(s))|^2\Big]\d s+
\int_0^{\t}\mathbb{E}\left[\varepsilon\int_Z|g^n(\u(s),z)|^2\lambda(dz)\right]\d
s.
\end{align*}
Then we use the hypothesis (H.2) to obtain
\begin{align*}
&\mathbb{E}\left[|\u(\t)|^2 +
\nu \int_0^{\t}\|\u(s)\|^2 \d s\right] \\
&\quad\leq \mathbb{E}|u(0)|^2 +
\frac{1}{\nu}\int_0^{\t}\|f(s)\|^2_{V^{\prime}} \d s  + \varepsilon
K\int_0^{\t}\mathbb{E}\Big(1 + |\u(s)|^2\Big)\d s.
\end{align*}
So finally we obtain
\begin{align*}
&\mathbb{E}\Big[|\u(\t)|^2\ +
\nu \int_0^{\t}\|\u(s)\|^2 \d s\Big] \\
&\quad\leq \mathbb{E}|u(0)|^2 +
\frac{1}{\nu}\int_0^{\t}\|f(s)\|^2_{V^{\prime}} \d s  + \varepsilon
KT + \varepsilon K\int_0^{\t}\mathbb{E}\Big( |\u(s)|^2\Big)\d s.
\end{align*}
In particular
\begin{align*}
&\mathbb{E}\left[|\u(\t)|^2 \right]\\
&\quad\leq \mathbb{E}|u(0)|^2 +
\frac{1}{\nu}\int_0^{\t}\|f(s)\|^2_{V^{\prime}} \d s  + \varepsilon
KT + \varepsilon K\int_0^{\t}\mathbb{E}\Big( |\u(s)|^2\Big)\d s.
\end{align*}
Applying Gronwall's Inequality, we obtain
\begin{align*}
&\mathbb{E}\left[|\u(\t)|^2 \right]\leq e^{\varepsilon
KT}\left[\mathbb{E}|u(0)|^2 +
\frac{1}{\nu}\int_0^{\t}\|f(s)\|^2_{V^{\prime}} \d s + \varepsilon
KT\right].
\end{align*}
So we get
\begin{align*}
&\mathbb{E}\left[|\u(\t)|^2\right] +
\nu \int_0^{\t}\mathbb{E}\|\u(s)\|^2 \d s \\
&\;\leq \Big(1 + \varepsilon KTe^{\varepsilon
KT}\Big)\Big(\mathbb{E}|u(0)|^2 +
\frac{1}{\nu}\int_0^{\t}\|f(s)\|_{V^\prime}^2ds + \varepsilon
KT\Big).
\end{align*}
Taking the limit as $N \rightarrow \infty$ we have the result
\eqref{energy1}.

\sni To prove \eqref{energy2}, we proceed in the similar way as
above, but we take supremum upto time $\T$ before taking the
expectation in equation \eqref{555},
\begin{align}
&\mathbb{E}\left[\sup_{0\leq t\leq \T}|\u(t)|^2 + \nu \int_0^{\T}\|\u(t)\|^2 \d t\right]\nonumber \\
&\quad\leq \mathbb{E}|u(0)|^2 +
\frac{1}{\nu}\int_0^{\T}\|f(t)\|^2_{V^{\prime}} \d t + \varepsilon K
T
+ \varepsilon K\mathbb{E}\int_0^{\T}\sup_{0\leq s\leq t}|\u(s)|^2\d t \nonumber \\
&\quad\quad + 2\sqrt{\varepsilon} \mathbb{E}\left[\sup_{0\leq t\leq \T}\left|\int_0^t(\sigma_n(s, \u(s)), \u(s)) \d W_n(s)\right|\right] \nonumber \\
&\quad\quad + 2\varepsilon\mathbb{E}\left[\sup_{0\leq t\leq
\T}\left|\int_0^t\int_Z\big(\u(s-),g^n(\u(s-),z)\big)\tilde{N}(ds,dz)\right|\right].\label{3}
\end{align}
Applying Burkholder-Davis-Gundy inequality, condition (H.2) and
Young's inequality to the term $$2\sqrt{\varepsilon}
\mathbb{E}\left[\sup_{0\leq t\leq \T}\left|\int_0^t(\sigma_n(s,
\u(s)), \u(s)) \d W_n(s)\right|\right]$$ we get,
\begin{align}
&2\sqrt{\varepsilon} \mathbb{E}\left[\sup_{0\leq t\leq \T}\left|\int_0^t(\sigma_n(s, \u(s)), \u(s)) \d W_n(s)\right|\right]\nonumber\\
&\quad\leq 2\sqrt{2\varepsilon}  \mathbb{E}\left[\int_0^{\T}|\sigma(s,\u(s))|^2|\u(s)|^2\d s\right]^{1/2}\nonumber\\
&\quad\leq 2\sqrt{2\varepsilon K}  \mathbb{E}\left[\left(\int_0^{\T} \left(1+|\u(t)|^2\right) |\u(t)|^2 \d t \right)^{1/2}\right]\nonumber\\
&\quad\leq 2\sqrt{2\varepsilon K} \mathbb{E}\left[\sup_{0\leq t\leq \T} |\u(t)| \left(\int_0^{\T} \left(1+|\u(t)|^2\right) \d t\right)^{1/2}\right]\nonumber\\
&\quad\quad \left[\textrm{Young's inequality }ab\leq \eta a^2 +
C(\eta)b^2\quad (\eta>0)\right. \nonumber\\&\left.\quad\quad\textrm{
for } C(\eta) =\frac{1}{4\eta}, \textrm{ by
taking }\eta =\frac{1}{8\sqrt{2\varepsilon K}},C(\eta) = 2\sqrt{2\e K}\right]\nonumber\\
&\quad\leq \frac{1}{4}\mathbb{E} \left(\sup_{0\leq t\leq \T} |\u(t)|^2\right) + 8\varepsilon K\mathbb{E}\int_0^{\T} |\u(t)|^2 \d t + 8\varepsilon KT\nonumber\\
&\quad\leq \frac{1}{4}\mathbb{E}\left(\sup_{0\leq t\leq \T}
|\u(t)|^2\right) + 8\varepsilon K\mathbb{E}\int_0^{\T} \sup_{0\leq
s\leq t}|\u(s)|^2 \d t + 8\varepsilon KT.\label{4}
\end{align}
Now again by applying Burkholder-Davis-Gundy inequality, condition
(H.2) and Young's inequality to the term
$$2\varepsilon\mathbb{E}\left[\sup_{0\leq t \leq {\T}}\left|\int_0^t\int_Z\big(\u(s-),g^n(\u(s-),z)\big)\tilde{N}(ds,dz)\right|\right]$$ we get,
\begin{align*}
&2\varepsilon\mathbb{E}\left[\sup_{0\leq t \leq {\T}}\left|\int_0^t\int_Z\big(\u(s-),g^n(\u(s-),z)\big)\tilde{N}(ds,dz)\right|\right]\nonumber\\
&\quad\leq 2\sqrt{2}\varepsilon\mathbb{E}\left[\int_0^{\T}\int_Z\left|\big(\u(s),g^n(\u(s),z)\big)\right|^2\lambda(dz)ds\right]^{1/2}\nonumber\\
&\quad\leq 2\sqrt{2}\varepsilon\mathbb{E}\left[\int_0^{\T}\int_Z|\u(s)|^2|g^n(\u(s),z)|^2\lambda(dz)ds\right]^{1/2}\nonumber\\
&\quad\leq 2\sqrt{2}\varepsilon\mathbb{E}\left[\sup_{0\leq t \leq {\T}}|\u(t)|\left(\int_0^{\T}\int_Z|g^n(\u(s),z)|^2\lambda(dz)ds\right)^{1/2}\right]\nonumber\\
&\quad\leq 2\sqrt{2}\varepsilon\mathbb{E}\left[\sup_{0\leq t \leq {\T}}|\u(t)|\left(\int_0^{\T}K(1 + |\u(s)|^2)\d s\right)^{1/2}\right]\nonumber\\
&\quad\quad \left[\textrm{Young's inequality }ab\leq \eta a^2 +
C(\eta)b^2\quad (\eta>0)\right. \nonumber\\&\left.\quad\quad\textrm{
for } C(\eta) =\frac{1}{4\eta}, \textrm{ by
taking }\eta =\frac{1}{8\sqrt{2}\varepsilon },C(\eta) = 2\sqrt{2}\e \right]\nonumber\\
\end{align*}
\begin{align}
&2\varepsilon\mathbb{E}\left[\sup_{0\leq t \leq {\T}}\left|\int_0^t\int_Z\big(\u(s-),g^n(\u(s-),z)\big)\tilde{N}(ds,dz)\right|\right]\nonumber\\
&\quad\leq \frac{1}{4}\mathbb{E}\left[\sup_{0\leq t \leq {\T}}|\u(t)|^2\right] + 8\varepsilon^2 K \mathbb{E}\int_0^{\T}|\u(t)|^2\d t + 8\varepsilon^2 KT\nonumber\\
&\quad\leq \frac{1}{4}\mathbb{E}\left[\sup_{0\leq t \leq
{\T}}|\u(t)|^2\right] + 8\varepsilon^2
K\mathbb{E}\int_0^{\T}\sup_{0\leq s\leq t}|\u(s)|^2\d t +
8\varepsilon^2 KT.\label{5}
\end{align}
Replace \eqref{4} and \eqref{5} in \eqref{3},
\begin{align}\label{100}
\mathbb{E}\left[\sup_{0\leq t\leq \T}|\u(t)|^2\right] &+ 2\nu \int_0^{\T}\mathbb{E}\|\u(t)\|^2 \d t\nonumber \\
&\leq 2\mathbb{E}|u(0)|^2 + \frac{2}{\nu}\int_0^{\T}\|f(t)\|^2_{V^{\prime}} \d t + 2\varepsilon KT(9+8\varepsilon) \nonumber \\
&\quad+2\varepsilon K (9+8\varepsilon)
\mathbb{E}\int_0^{\T}\sup_{0\leq s\leq t}|\u(s)|^2\d t.
\end{align}
Note $\T\rightarrow T$ a.s. as $N\rightarrow\infty$. Thus taking the limit in the above estimate \eqref{100} as $N\rightarrow\infty$, one can get for all $\varepsilon$
\begin{align}\label{1000}
\mathbb{E}\left[\sup_{0\leq t\leq T}|\u(t)|^2\right] &+ 2\nu \int_0^{T}\mathbb{E}\|\u(t)\|^2 \d t\nonumber \\
&\leq 2\mathbb{E}|u(0)|^2 + \frac{2}{\nu}\int_0^{T}\|f(t)\|^2_{V^{\prime}} \d t + 2\varepsilon KT(9+8\varepsilon) \nonumber \\
&\quad+ 2\varepsilon K(9+8\varepsilon)
\mathbb{E}\int_0^{T}\sup_{0\leq s\leq t}|\u(s)|^2\d t.
\end{align}
In particular
\begin{align}
\mathbb{E}\Big[\sup_{0\leq t\leq T}|\u(t)|^2\Big] &\leq
2\mathbb{E}|u(0)|^2 + \frac{2}{\nu}\int_0^{T}\|f(t)\|^2_{V^{\prime}}
\d t + 2\varepsilon KT(9+8\varepsilon) \nonumber \\&\quad +
2\varepsilon K(9+8\varepsilon) \mathbb{E}\int_0^{T}\sup_{0\leq s\leq
t}|\u(s)|^2\d t.
\end{align}
Now by applying Gronwall's Inequality, we obtain
\begin{align}
&\mathbb{E}\Big[\sup_{0\leq t\leq T}|\u(t)|^2\Big] \nonumber \\
&\leq e^{2\varepsilon KT(9+8\varepsilon)}\left[2\mathbb{E}|u(0)|^2 +
\frac{2}{\nu}\int_0^T\|f(t)\|^2_{V^{\prime}} \d t + 2\varepsilon
KT(9+8\varepsilon)\right].\label{110}
\end{align}
Now using by \eqref{110} in \eqref{1000} one can deduce that
\begin{align}
\mathbb{E}\left[\sup_{0\leq t\leq T}|\u(t)|^2\right] &+ 2\nu \int_0^{T}\mathbb{E}\|\u(t)\|^2 \d t \nonumber \\
&\leq \left(1 + 2\varepsilon KT(9+8\varepsilon)e^{2\varepsilon
KT(9+8\varepsilon)}\right)\nonumber\\&\quad\left(2\mathbb{E}|u(0)|^2
+ \frac{2}{\nu}\int_0^T\|f(t)\|^2_{V^{\prime}} \d t+ 2\varepsilon
KT(9+8\varepsilon)\right).
\end{align}
Hence, we obtain
\begin{align*}
&\mathbb{E}\left[\sup_{0\leq t\leq T}|\u(t)|^2\right] + 2\nu
\int_0^{T}\mathbb{E}\|\u(t)\|^2 \d t \leq C\left(\mathbb{E}|u(0)|^2,
\int_0^T\|f(t)\|^2_{V^{\prime}}\d t,\nu, T\right).
\end{align*}
\end{proof}

\begin{theorem}\label{energya}
Let $f$ be in $\mathrm{L}^2( [0, T], H)$,  $u (0)$ be $\mathcal{F}_0$ measurable,
$\sigma \in C([0, T] \times V; L_Q(H_0; H))$,  $g\in \mathbb{H}^2_{\lambda}([0, T] \times Z; H)$ and $\mathbb{E}|u (0)|^2 < \infty$.
Let $\u$ denote the unique strong solution of the stochastic differential equation \eqref{variational}
in $\mathcal{D}([0, T], H_n)$. Then with $K$ as in condition (H.2), the following estimates hold:\\
For any $\delta > 0$,
\begin{align}\label{energy3}
\mathbb{E}&|\u(t)|^2 e^{-\delta t} + 2\nu\int_0^T \mathbb{E}\|\u(t)\|^2 e^{-\delta t} \d t \nonumber\\
&\leq \left(1 + \varepsilon KTe^{\varepsilon
KT}\right)\left(\mathbb{E}|u(0)|^2 +
\frac{1}{\delta}\int_0^T|f(t)|^2e^{-\delta t}dt + \frac{\varepsilon
K}{\delta}\right),
\end{align}
and for any $\delta >0 $,
\begin{align}\label{energy4}
\mathbb{E}\left[\sup_{0\leq t\leq T} |\u(t)|^2e^{-\delta t}\right] &+ 4\nu\int_0^T\mathbb{E}\|\u(t)\|^2e^{-\delta t}dt \nonumber\\
&\leq C\left(\mathbb{E}|u(0)|^2, \int_0^T|f(t)|^2e^{-\delta t}dt,
\delta, T\right).
\end{align}
\end{theorem}

\begin{proof}
In order prove this theorem we use the same method as in the previous theorem and also use the same stopping time argument. \\
We consider the function $e^{-\delta t} |\u(t)|^2$ for $\delta>0$ and apply the It\^{o} Lemma to get,
\begin{align}
&\d \left[|\u(t)|^2 e^{-\delta t}\right] + 2\nu\|\u(t)\|^2e^{-\delta t}\d t + \delta |\u(t)|^2 e^{-\delta t}\d t\nonumber \\
&\quad=\left[ 2(f(t),
\u(t)) + \varepsilon \Tr (\sigma_n(t,\u(t)) Q \sigma_n(t,\u(t)))\right]e^{-\delta t}\d t\nonumber\\
&\quad\quad + 2\sqrt{\varepsilon} (\sigma_n(t, \u(t)), \u(t))e^{-\delta t} \d W_n(t) \nonumber\\&\quad\quad+ e^{-\delta t}\varepsilon\int_Z|g^n(\u(t-),z)|^2 N(dt,dz) \nonumber \\
&\quad\quad+ 2e^{-\delta
t}\int_Z\varepsilon\left(\u(t-),g^n(\u(t-),z)\right)\tilde{N}(dt,dz).\label{6}
\end{align}
Note that $$2(f(t),\u(t)) \leq \delta |\u(t)|^2 + \frac{1}{\delta} |f(t)|^2.$$
So from the above relation we get
\begin{align}
&\d \left[|\u(t)|^2 e^{-\delta t}\right] + 2\nu\|\u(t)\|^2e^{-\delta t}\d t\nonumber \\
&\quad\leq\frac{1}{\delta}|f(t)|^2e^{-\delta t}\d t + \varepsilon \Tr (\sigma_n(t,\u(t)) Q \sigma_n(t,\u(t)))e^{-\delta t}\d t\nonumber\\
&\quad\quad + 2\sqrt{\varepsilon} (\sigma_n(t, \u(t)), \u(t))e^{-\delta t} \d W_n(t) \nonumber\\
&\quad\quad+ e^{-\delta t}\int_Z\varepsilon|g^n(\u(t-),z)|^2 N(dt,dz) \nonumber \\
&\quad\quad+ 2e^{-\delta
t}\int_Z\varepsilon\left(\u(t-),g^n(\u(t-),z)\right)\tilde{N}(dt,dz).\label{7}
\end{align}
Hence upon writing \eqref{6} in the integral form, then taking expectation and proceeding as in the previous stopping time given in the proof of Theorem \ref{energy} one can get
\begin{align*}
& \mathbb{E} |\u(t)|^2 e^{-\delta t} + 2\nu \mathbb{E}\int_0^T\|\u(t)\|^2e^{-\delta t}\d t\nonumber \\
&\quad\leq\mathbb{E}|u(0)|^2 + \frac{1}{\delta}\int_0^T|f(t)|^2e^{-\delta t}\d t + \mathbb{E}\int_0^T\varepsilon|\sigma_n(t,\u(t))|^2e^{-\delta t}\d t\nonumber\\
&\quad\quad + \mathbb{E}\int_0^Te^{-\delta
t}\int_Z\varepsilon|g^n(\u(t),z)|^2 \lambda(dz)\d t.
\end{align*}
Since the terms
$$2\sqrt{\varepsilon}\int_0^T (\sigma_n(t, \u(t)), \u(t))e^{-\delta t} \d W_n(t)$$
and
$$2\int_0^T  e^{-\delta t}\varepsilon\int_Z\left(\u(t-),g^n(\u(t-),z)\right)\tilde{N}(dt,dz)$$
are martingales and having zero averages. Now applying (H.2) one can
obtain
\begin{align}
& \mathbb{E} |\u(t)|^2 e^{-\delta t} + 2\nu \mathbb{E}\int_0^T\|\u(t)\|^2e^{-\delta t}\d t\nonumber \\
&\quad\leq\mathbb{E}|u(0)|^2 + \frac{1}{\delta}\int_0^T|f(t)|^2e^{-\delta t}\d t + \frac{\varepsilon K}{\delta}
+ \varepsilon K\int_0^T\mathbb{E}|\u(t)|^2e^{-\delta t}\d t.\label{8}
\end{align}
In particular
\begin{align*}
& \mathbb{E} |\u(t)|^2 e^{-\delta t} \nonumber \\
&\quad\leq\mathbb{E}|u(0)|^2 + \frac{1}{\delta}\int_0^T|f(t)|^2e^{-\delta t}\d t + \frac{\varepsilon K}{\delta}
+ \varepsilon K\int_0^T\mathbb{E}|\u(t)|^2e^{-\delta t}\d t.
\end{align*}
Applying Gronwall's Inequality we get,
\begin{align*}
& \mathbb{E} |\u(t)|^2 e^{-\delta t}\leq e^{\varepsilon
KT}\left[\mathbb{E}|u(0)|^2 +
\frac{1}{\delta}\int_0^T|f(t)|^2e^{-\delta t}\d t +
\frac{\varepsilon K}{\delta}\right].
\end{align*}
By using above relation in \eqref{8} one can deduce that
\begin{align}
& \mathbb{E} |\u(t)|^2 e^{-\delta t} + 2\nu \int_0^T\mathbb{E}\|\u(t)\|^2e^{-\delta t}\d t\nonumber \\
&\quad\leq\left(1+\varepsilon KTe^{\varepsilon
KT}\right)\left(\mathbb{E}|u(0)|^2 +
\frac{1}{\delta}\int_0^T|f(t)|^2e^{-\delta t}\d t +
\frac{\varepsilon K}{\delta} \right).\label{90}
\end{align}
This proves \eqref{energy3}.

\sni Now for getting \eqref{energy4} we proceed as above and taking
supremum  before taking the expectaion, in \eqref{7}
\begin{align}
&\mathbb{E} \left[\sup_{0\leq t\leq T}|\u(t)|^2 e^{-\delta t}\right] + 2\nu \int_0^T\mathbb{E}\|\u(t)\|^2e^{-\delta t}\d t\nonumber \\
&\leq\mathbb{E}|u(0)|^2 + \frac{1}{\delta}\int_0^T|f(t)|^2e^{-\delta t}\d t + \mathbb{E}\sup_{0\leq s\leq T}\int_0^s\varepsilon|\sigma_n(t,\u(t)|^2e^{-\delta t}\d t\nonumber\\
&\quad + 2\sqrt{\varepsilon}\mathbb{E}\sup_{0\leq s\leq T}\int_0^s(\sigma_n(t,\u(t)),\u(t))e^{-\delta t}\d W_n(t) \nonumber\\
&\quad + \mathbb{E}\sup_{0\leq s\leq T}\int_0^s e^{-\delta t}\int_Z\varepsilon|g^n(\u(t-),z)|^2\lambda(dz)\d t \nonumber \\
&\quad + 2\mathbb{E}\sup_{0\leq s\leq T}\int_0^s e^{-\delta t}\int_Z\varepsilon\left(\u(t-),g^n(\u(t-),z)\right)\mathbb{}\tilde{N}(dt,dz)\nonumber \\
&\leq\mathbb{E}|u(0)|^2 + \frac{1}{\delta}\int_0^T|f(t)|^2e^{-\delta t}\d t + \varepsilon K\mathbb{E}\left[\int_0^T\sup_{0\leq s\leq t} |\u(s)|^2e^{-\delta t} \d t\right] + \frac{\varepsilon K}{\delta} \nonumber \\
&\quad + 2\sqrt{\varepsilon}\mathbb{E}\sup_{0\leq s\leq T}\left|\int_0^s(\sigma_n(t,\u(t)),\u(t))e^{-\delta t}\d W_n(t)\right| \nonumber \\
&\quad + 2\varepsilon\mathbb{E}\sup_{0\leq s\leq T}\left|\int_0^s
e^{-\delta
t}\int_Z\left(\u(t-),g^n(\u(t-),z)\right)\tilde{N}(dt,dz)\right|.\label{9}
\end{align}
Next we consider $$2\sqrt{\varepsilon}\mathbb{E}\left[\sup_{0\leq
s\leq T}\left|\int_0^s(\sigma_n(t,\u(t)),\u(t))e^{-\delta t}\d
W_n(t)\right|\right]$$ and applying Burkholder-Davis-Gundy
Inequality, Young's Inequality and condition $(H.2)$, we get
\begin{align}
&2\sqrt{\varepsilon}\mathbb{E}\left[\sup_{0\leq s\leq T}\left|\int_0^s(\sigma_n(t,\u(t)),\u(t))e^{-\delta t}\d W_n(t)\right|\right]\nonumber \\
\quad&\leq2\sqrt{2\varepsilon}\mathbb{E}\left[\int_0^T|\sigma_n(t,\u(t)|^2|\u(t)|^2e^{-2\delta t} \d t\right]^{1/2}\nonumber \\
\quad&\leq2\sqrt{2\varepsilon K}\mathbb{E}\left[\int_0^T\Big(1 + |\u(t)|^2\Big)|\u(t)|^2e^{-2\delta t}\d t\right]^{1/2}\nonumber \\
\quad&\leq2\sqrt{2\varepsilon K}\mathbb{E}\left[\left(\sup_{0\leq t\leq T}|\u(t)|e^{-\delta t/2}\right)\left(\int_0^T\big(1 + |\u(t)|^2\big)e^{-\delta t} \d t\right)^{1/2}\right]\nonumber \\
\quad&\leq\frac{1}{4}\mathbb{E}\Big[\sup_{0\leq t\leq T}|\u(t)|^2e^{-\delta t}\Big] + 8\varepsilon K \mathbb{E}\int_0^T|\u(t)|^2e^{-\delta t}\d t + \frac{8\varepsilon K}{\delta}\nonumber \\
\quad&\leq\frac{1}{4}\mathbb{E}\Big[\sup_{0\leq t\leq T}|\u(t)|^2e^{-\delta t}\Big] + 8\varepsilon K \mathbb{E}\int_0^T\sup_{0\leq s\leq t}|\u(s)|^2e^{-\delta t}\d t + \frac{8\varepsilon K}{\delta}.\label{10}
\end{align}
Now again applying Burkholder-Davis-Gundy Inequality, Young's Inequality and (H.2) to the term
$$2\varepsilon \mathbb{E}\left[\sup_{0\leq s\leq T}\left|\int_0^se^{-\delta t}\int_Z\left(\u(t-),g^n(\u(t-),z)\right)\mathbb{}\tilde{N}(dt,dz)\right|\right]$$
we get,
\begin{align}
&2\varepsilon\mathbb{E}\left[\sup_{0\leq s\leq T}\left|\int_0^se^{-\delta t}\int_Z\left(\u(t),g^n(\u(t-),z)\right)\mathbb{}\tilde{N}(dt,dz)\right|\right]\nonumber \\
&\leq2\sqrt{2}\varepsilon\mathbb{E}\left(\int_0^T\int_Z\left|\left(\u(t),g^n(\u(t),z)\right)e^{-\delta t}\right|^2\lambda(dz)dt\right)^{1/2}\nonumber \\
&\leq 2\sqrt{2}\varepsilon\mathbb{E}\left(\int_0^T\int_Z|g^n(\u(t),z)|^2 |\u(t)|^2e^{-2\delta t}\lambda(dz)\d t\right)^{1/2}\nonumber \\
&\leq 2\sqrt{2}\varepsilon\mathbb{E}\left(\int_0^T K\left(1 + |\u(t)|^2\right)|\u(t)|^2e^{-2\delta t}\lambda(dz)\d t\right)^{1/2}\nonumber \\&\leq 2\sqrt{2}\varepsilon\mathbb{E}\left[\left(\sup_{0\leq t\leq T}|\u(t)|e^{-\delta t/2}\right)\left(\int_0^TK\left(1 + |\u(t)|^2\right)e^{-\delta t}\d t\right)^{1/2}\right]\nonumber \\
&\leq \frac{1}{4}\mathbb{E}\Big[\sup_{0\leq t\leq T}|\u(t)|^2e^{-\delta t}\Big] + 8\varepsilon^2 K\mathbb{E}\int_0^T\sup_{0\leq s\leq t}|\u(s)|^2e^{-\delta t}\d t + \frac{8\varepsilon^2 K}{\delta}.\label{11}
\end{align}
By applying \eqref{10} and \eqref{11} in \eqref{9} one can deduce that
\begin{align}
\mathbb{E}\left[\sup_{0\leq t\leq T}|\u(t)|^2e^{-\delta t}\right] &+ 4\nu\int_0^T\mathbb{E}\|\u(t)\|^2e^{-\delta t}\d t
\nonumber \\ &\leq2\mathbb{E}|u(0)|^2 + \frac{2}{\delta}\int_0^T|f(t)|^2e^{-\delta t}\d t +
\frac{2\varepsilon K(9+8\varepsilon)}{\delta} \nonumber \\&\quad+ 2\varepsilon K(9+8\varepsilon)\mathbb{E}\int_0^T\sup_{0\leq s\leq t}|\u(s)|^2e^{-\delta t}\d t .\label{12}
\end{align}
From the above expression one can write,
\begin{align*}
\mathbb{E}\left[\sup_{0\leq t\leq T}|\u(t)|^2e^{-\delta t}\right]&\leq2\mathbb{E}|u(0)|^2 + \frac{2}{\delta}\int_0^T|f(t)|^2e^{-\delta t}\d t \nonumber \\ & +\frac{2\varepsilon K(9+8\varepsilon)}{\delta} + 2\varepsilon K(9+8\varepsilon)\mathbb{E}\int_0^T\sup_{0\leq s\leq t}|\u(s)|^2e^{-\delta t}\d t .
\end{align*}
Applying Gronwall's Inequality
\begin{align}
&\mathbb{E}\left[\sup_{0\leq t\leq T}|\u(t)|^2e^{-\delta t}\right]\nonumber \\
&\leq e^{2\varepsilon KT(9+8\varepsilon)}\Big[2\mathbb{E}|u(0)|^2 +
\frac{2}{\delta}\int_0^T|f(t)|^2e^{-\delta t}\d t +
\frac{2\varepsilon K(9+8\varepsilon)}{\delta}\Big].\label{13}
\end{align}
Using \eqref{13} in \eqref{12} we get
\begin{align}
\mathbb{E}\left[\sup_{0\leq t\leq T}|\u(t)|^2e^{-\delta t}\right] &+ 4\nu\int_0^T\mathbb{E}\|\u(t)\|^2e^{-\delta t}\d t\nonumber \\
&\leq\left(1 + 2\varepsilon KT(9+8\varepsilon)e^{2\varepsilon
KT(9+8\varepsilon)}\right)\nonumber
\\&\quad\left(2\mathbb{E}|u(0)|^2 +
\frac{2}{\delta}\int_0^T|f(t)|^2e^{-\delta t}\d t +
\frac{2\varepsilon K(9+8\varepsilon)}{\delta}\right) .\label{14}
\end{align}
By using the above we get the required result \eqref{energy4}
\end{proof}

\begin{definition}($Strong\ Solution$)
A strong solution $\ue$ is defined on a given probability
space $(\Omega, \mathcal{F}, \mathcal{F}_{t}, P)$ as a
$\mathrm{L}^2(\Omega;\mathrm{L}^{\infty
}(0,T; H)\cap
\mathrm{L}^2(0,T; V)\cap
\mathcal{D}(0,T; H))$ valued adapted process which
satisfies the stochastic GOY model
\begin{align}\label{15}
&\d \ue + \big[\nu A\ue + B(\ue, \ue)\big] \d t = f(t) \d t + \sqrt{\e}\sigma(t, \ue) \d W(t) + \varepsilon\int_Zg(\ue,z)\tilde{N}(dt,dz)\\
& \ue(0) = u_0,\nonumber
\end{align}
 in the
weak sense and also the energy inequalities in Theorem \ref{energy} and Theorem\ref{energya}.
\end{definition}


\begin{theorem}\label{existence}
Let  $u (0)$ be $\mathcal{F}_0$ measurable and $\ \mathbb{E}|u_0|^2
< \infty.$ Let $f\in \mathrm{L}^2(0, T; V^{\prime}).$ We also assume
that $0 < \e < \frac{\nu}{C}$ and the diffusion coefficient
satisfies the conditions (H.1)-(H.3). Then there exists unique
adapted process $\ue(t, x, w)$ with the regularity
$$\ue \in \mathrm{L}^2(\Omega; \mathcal{D}(0, T; H) \cap \mathrm{L}^2(0, T; V))$$
satisfying the stochastic GOY model \eqref{15} and the a priori
bounds in Theorem \ref{energy} and Theorem \ref{energya} .
\end{theorem}

\begin{proof} Part I(Existence)\\
Using the a priori estimate in the Theorem \ref{energy} and Theorem
\ref{energya}, it follows from the Banach-Alaoglu theorem that along
a subsequence, the Galerkin approximations $\{\u\}$ have the
following limits:
\begin{align}
&\u\longrightarrow \ue\quad  \text {weak star in}\ \mathrm{L}^2(\Omega ;
  \mathrm{L}^{\infty}(0,T; H)) \cap\mathrm{L}^{2}
  (\Omega;\mathrm{L}^{2}(0,T; V)),\nonumber\\
& F(\u)\longrightarrow F^{\e}_0\quad \text{weakly in}\
\mathrm{L}^{2}(\Omega;\mathrm{L}^{2}(0,T; V^{\prime})),\nonumber\\
& \sigma_n(\cdot, \u) \longrightarrow S^{\e}\quad \text{weakly in}\ \mathrm{L}^{2}(\Omega;\mathrm{L}^{2}(0,T; \mathrm{L}_Q))\nonumber \\
&g^n(\u,\cdot)\longrightarrow G^\varepsilon\quad \text{weakly in}\
\mathbb{H}^2_{\lambda}([0, T] \times Z; H).\label{16}
\end{align}
The assertion of the second statement holds since $F(\u)$ is bounded
in \\ $\mathrm{L}^{2}(\Omega;\mathrm{L}^{2}(0,T; V^{\prime}))$.
Likewise since diffusion coefficient has the linear growth property
and $\u$ is bounded in $\mathrm{L}^2(0, T; V)$ uniformly in $n$, the
last two statements hold. Then $\ue$ has the It\^{o} differential
\begin{align*}
\d \ue(t) = F^{\e}_0(t)\d t &+ \sqrt{\e} S^{\e}(t)\d
W(t)+\varepsilon\int_ZG^\varepsilon(t)\tilde{N}(dt,dz)\nonumber \\
&\quad \text{weakly in}\
\mathrm{L}^{2}(\Omega;\mathrm{L}^{2}(0,T;V^\prime)).
\end{align*}
Let us set,
\begin{equation}\label{17}
r(t):= \frac{2}{\nu^3}\int_0^t \|\ve(s)\|_{\mathrm{L}^4}^4 \d s,
\end{equation}
where $\ve(t, x, \omega)$ is any adapted process in
$\mathrm{L}^{\infty}(\Omega \times (0, T); H)$. Here we suppress the
dependence of $\varepsilon$ in the notation of $r$ to make it easier
to read. Then applying the It\^{o} Lemma to the function $2e^{-r(t)}
|\u(t)|^2$, one obtains
\begin{align*}
\d \left[e^{-r(t)} |\u(t)|^2\right] &= e^{-r(t)}\big(2F(\u(t)) - \dot{r}(t) \u(t), \u(t)\big)\d t \\
&\quad + \e e^{-r(t)} |\sigma_n(t, \u(t))|^2_{\mathrm{L}_Q}\d t
\\ & \quad+ 2\sqrt{\e} e^{-r(t)}\big(\sigma_n(t, \u(t)), \u(t)\big)\d W(t) \\ & \quad + e^{-r(t)}\epsilon\int_Z\big|g^n(\u(t-),z)\big|^2N(dt,dz) \\
& \quad+
2e^{-r(t)}\varepsilon\int_Z\left(\u(t-),g^n(\u(t-),z)\right)\tilde{N}(dt,dz).
\end{align*}
Integrating between $0 \leq t\leq T$ and taking expectation,
\begin{align*}
&\mathbb{E}\left[e^{-r(T)} |\u(T)|^2 - |\u(0)|^2\right]\\
&\quad= \mathbb{E}\left[\int_0^T e^{-r(t)}\left(2F(\u(t)) - \dot{r}(t) \u(t), \u(t)\right)\d t\right]\\
&\quad\quad + \mathbb{E}\int_0^T e^{-r(t)} \e|\sigma_n(t, \u(t))|^2_{\mathrm{L}_Q}\d t\\
&\quad\quad + 2\sqrt{\e} \mathbb{E}\int_0^T e^{-r(t)}\big(\sigma_n(t, \u(t)), \u(t)\big)\d W(t)\\
&\quad\quad + \mathbb{E}\varepsilon\int_0^T e^{-r(t)}\int_Z\big|g^n(\u(t),z)\big|^2\lambda(dz)\d t\\
&\quad\quad +
2\mathbb{E}\int_0^Te^{-r(t)}\varepsilon\int_Z\left(\u(t-),g^n(\u(t-),z)\right)\tilde{N}(dt,dz).
\end{align*}
But the terms
$$2\sqrt{\varepsilon}\int_0^T e^{-r(t)}\big(\sigma_n(t, \u(t)), \u(t)\big)\d W(t)$$
and
$$2\int_0^Te^{-r(t)}\varepsilon\int_Z\left(\u(t-),g^n(\u(t-),z)\right)\tilde{N}(dt,dz)$$
are martingales and having zero averages. Hence we get
\begin{align*}
&\mathbb{E}\left[e^{-r(T)} |\u(T)|^2 - |\u(0)|^2\right]\\
&\quad= \mathbb{E}\left[\int_0^T e^{-r(t)}\big(2F(\u(t)) - \dot{r}(t) \u(t), \u(t)\big)\d t\right]\\
&\quad\quad + \mathbb{E}\int_0^T e^{-r(t)} \e|\sigma_n(t, \u(t))|^2_{\mathrm{L}_Q}\d t\\
&\quad\quad + \mathbb{E}\int_0^T e^{-r(t)}\varepsilon\int_Z\big|g^n(\u(t),z)\big|^2\lambda(dz)\d t\\
\end{align*}
Then by the lower semi-continuity property of the weak convergence,
\begin{align*}
&\liminf_n \mathbb{E}\left[\int_0^T e^{-r(t)}\big(2F(\u(t)) - \dot{r}(t) \u(t), \u(t)\big)\d t \nonumber \right.\\
&\left.\quad + \int_0^T e^{-r(t)} \e|\sigma_n(t, \u(t))|^2_{\mathrm{L}_Q}\d t + \int_0^T e^{-r(t)}\varepsilon\int_Z\big|g^n(\u(t),z)\big|^2\lambda(dz)\d t\right]\nonumber\\
&\quad\quad = \liminf_n \mathbb{E}\left[e^{-r(T)} |\u(T)|^2 - |\u(0)|^2\right]\nonumber\\
&\quad\quad \geq \mathbb{E}\left[e^{-r(T)} |\ue(T)|^2 - |\ue(0)|^2\right]\nonumber\\
&\quad\quad = \mathbb{E}\left[\int_0^T e^{-r(t)}\big(2F_0^{\e}(t) - \dot{r}(t) \ue(t), \ue(t)\big)\d t +\e \int_0^T e^{-r(t)} |S^{\e}|^2_{\mathrm{L}_Q}\d t\nonumber\right.\\ &\left.\quad\quad\quad\quad\quad\qquad\qquad + \int_0^Te^{-r(t)}\varepsilon\int_Z|G^\varepsilon|^2\lambda(dz)\d t\right]
\end{align*}
Hence we get
\begin{align}
&\liminf_n \mathbb{E}\left[\int_0^T e^{-r(t)}\big(2F(\u(t)) - \dot{r}(t) \u(t), \u(t)\big)\d t \nonumber \right.\\
&\left.\quad \quad + \int_0^T e^{-r(t)} \e|\sigma_n(t, \u(t))|^2_{\mathrm{L}_Q}\d t + \int_0^T e^{-r(t)}\varepsilon\int_Z\big|g^n(\u(t),z)\big|^2\lambda(dz)\d t\right]\nonumber\\
&\quad\quad\geq\mathbb{E}\left[\int_0^T e^{-r(t)}\big(2F_0^{\e}(t) - \dot{r}(t) \ue(t), \ue(t)\big)\d t +\e \int_0^T e^{-r(t)} |S^{\e}|^2_{\mathrm{L}_Q}\d t\nonumber\right.\\ &\left.\quad\quad\quad\quad\quad\qquad\qquad + \int_0^Te^{-r(t)}\varepsilon\int_Z|G^\varepsilon|^2\lambda(dz)\d t\right].\label{lsc}
\end{align}
Now by monotonicity property from Lemma \ref{Mon},
\begin{align*}
& 2\mathbb{E}\left[\int_0^T e^{-r(t)} \left(F(\u(t)) - F(\ve(t)), \u(t) - \ve(t)\right)\d t\right]  \\
&\quad - \mathbb{E}\left[\int_0^T e^{-r(t)} \dot{r}(t)|\u(t) - \ve(t)|^2 \d t\right]  \\
&\quad +  \mathbb{E}\left[\int_0^T e^{-r(t)}\e |\sigma_n(t, \u(t)) - \sigma_n(t, \ve(t))|^2_{\mathrm{L}_Q}\d t\right]\\
&\quad + \mathbb{E}\left[\int_0^Te^{-r(t)}\int_Z\varepsilon|g^n(\u(t),z)-g^n(\ve(t),z)|^2\lambda(dz)\d t\right]\\
&\quad \leq 0.
\end{align*}
Rearranging the terms,
\begin{align*}
& \mathbb{E}\left[\int_0^T e^{-r(t)}\big(2F(\u(t)) - \dot{r}(t) \u(t), \u(t)\big)\d t \right.\\
&\left.\quad +\int_0^T e^{-r(t)} \e |\sigma_n(t, \u(t))|^2_{\mathrm{L}_Q}\d t\ + \int_0^Te^{-r(t)}\int_Z\varepsilon|g^n(\u(t),z)|^2\lambda(dz)\d t\right]\\
&\quad\quad\leq \mathbb{E}\left[\int_0^T e^{-r(t)}\left(2F(\u(t))-\dot{r}(t)(2\u(t) - \ve(t)), \ve(t)\right) \d t\right] \\
&\quad\quad\quad + \mathbb{E}\left[\int_0^T e^{-r(t)}\left(2F(\ve(t)), \u(t)-\ve(t)\right) \d t\right] \\
&\quad\quad\quad +  \e \mathbb{E}\left[\int_0^T e^{-r(t)}\big(2\sigma_n(t, \u(t))-\sigma_n(t, \ve(t)), \sigma_n(t, \ve(t))\big)_{\mathrm{L}_Q}\d t\right]
\\&\quad\quad\quad + \varepsilon\mathbb{E}\left[\int_0^Te^{-r(t)}\int_Z\big(2g^n(\u(t),z)-g^n(\ve(t),z),g^n(\ve(t),z)\big)\lambda(dz)\d t\right].
\end{align*}
Taking limit in $n$, using the result from \eqref{lsc}, we get
\begin{align*}
&\mathbb{E}\left[\int_0^T e^{-r(t)}\big(2F_0^{\e}(t) - \dot{r}(t) \ue(t), \ue(t)\big)\d t +\e \int_0^T e^{-r(t)} |S^{\e}|^2_{\mathrm{L}_Q}\d t\nonumber\right.\\ &\left.\quad\quad + \int_0^Te^{-r(t)}\int_Z\varepsilon|G^\varepsilon|^2\lambda(dz)\d t\right] \\
&\quad\quad\ \leq \mathbb{E}\left[\int_0^Te^{-r(t)}\left(2F_0^{\e}(t)-\dot{r}(t)(2\ue(t) - \ve(t)), \ve(t)\right) \d t\right] \\
&\quad\quad\quad + \mathbb{E}\left[\int_0^T e^{-r(t)}\left(2F(\ve(t)), \ue(t)-\ve(t)\right) \d t\right] \\
&\quad\quad\quad +  \e \mathbb{E}\left[\int_0^T e^{-r(t)}\left(2S^{\e}(t)-\sigma(t, \ve(t)), \sigma(t, \ve(t))\right)_{\mathrm{L}_Q}\d t\right] \\&\quad\quad\quad + \e\mathbb{E}\left[\int_0^Te^{-r(t)}\int_Z\left(2G^{\e}(t)-g(\ve(t),z),g(\ve(t),z)\right)\lambda(dz)\d t\right].
\end{align*}
Rearranging the terms, we obtain
\begin{align*}
&\mathbb{E}\left[\int_0^T e^{-r(t)}\left(2F_0^{\e}(t)-2F(\ve(t)), \ue(t)-\ve(t)\right)\d t\right] \\
&\quad + \mathbb{E}\left[\int_0^T e^{-r(t)}\dot{r}(t) |\ue(t) - \ve(t)|^2 \d t\right]\\
&\quad + \e \mathbb{E}\left[\int_0^T e^{-r(t)} \|S(t)-\sigma(t, \ve(t))\|^2_{\mathrm{L}_Q}\d t\right]\\
&\quad + \varepsilon\mathbb{E}\left[\int_0^Te^{-r(t)}\int_Z\|G(t)-g(\ve(t),z)\|^2\lambda(dz)\d t\right]\\
&\leq 0.
\end{align*}
Notice that for $\ve=\ue$, $S(t)=\sigma(t, \ue(t))$ and $G(t)=g(\ue(t),z)$. Take $\ve = \ue - \mu\we$ with $\mu >0$ and $\we$ is an adapted process in $\mathrm{L}^2(\Omega; \mathcal{D}(0, T; H) \cap \mathrm{L}^2(0, T; V))$. Then,
\begin{align*}
&\mu \mathbb{E}\left[\int_0^T e^{-r(t)}\big(2F_0^{\e}(t)-2F(\ue - \mu\we)(t), \we(t)\big)\d t + \mu\int_0^T e^{-r(t)}\dot{r}(t)|\we(t)|^2\d t\right]\\
&\quad\leq 0.
\end{align*}
Dividing by $\mu$ on both sides of the inequality above and letting $\mu$ go to $0$, one obtains
\begin{align*}
\mathbb{E}\left[\int_0^T e^{-r(t)}\left(F_0^{\e}(t)-F(\ue(t)), \we(t)\right)\d t\right] \leq 0.
\end{align*}
Since $\we$ is arbitrary, we conclude that $F_0^{\e}(t)=F(\ue(t))$.
Thus the existence of the strong solution of the stochastic GOY model \eqref{15} has been proved.\\

\noindent Part II(Uniqueness)\\
If $\ve\in \mathrm{L}^2(\Omega; \mathcal{D}(0, T; H) \cap
\mathrm{L}^2(0, T; V))$ be another solution of the equation
\eqref{15} then $\we=\ue-\ve$ solves the stochastic differential
equation in $\mathrm{L}^2(\Omega;\mathrm{L}^2(0, T; V^\prime))$,
\begin{align}\label{17}
\d \we(t) = (F(\ue(t)) - F(\ve(t)))\d t &+ \sqrt{\e}(\sigma(t, \ue(t)) - \sigma(t, \ve(t))) \d W(t)\nonumber \\
&+\int_Z[g(\ue(t-),z)-g(\ve(t-),z)]\tilde{N}(dt,dz).
\end{align}
We denote $\sigma_d = \sigma(t, \ue(t)) - \sigma(t, \ve(t))$ and $g_d = g(\ue(t-),z)-g(\ve(t-),z)$.

We now apply It\^{o} Lemma to the function $2e^{-r(t)} |\we(t)|^2$,we get
\begin{align*}
\d \left[e^{-r(t)}  |\we(t)|^2\right] &= \left[-e^{-r(t)}\dot{r}(t)|\we(t)|^2 + 2e^{-r(t)}\left(F(\ue(t))-F(\ve(t)),\we(t)\right) \nonumber\right.\\
&\left.\quad + \e e^{-r(t)} \Tr(\sigma_d Q \sigma_d)\right]\d t + 2\sqrt{\e} e^{-r(t)} (\sigma_d, \we(t)) \d W(t)\nonumber \\
&\quad + e^{-r(t)}\varepsilon\int_Z|g_d|^2N(dt,dz) +
2e^{-r(t)}\varepsilon\int_Z\big(\we(t),g_d\big)\tilde{N}(dt,dz).
\end{align*}

Now using the local monotonicity of the sum of the linear and
nonlinear operators $A$ and $B$, e.g. equation \eqref{monotone2},we obtain
\begin{align*}
\d \left[e^{-r(t)}  |\we(t)|^2\right] & + \nu \|\we(t)\|^2e^{-r(t)}\d t \nonumber\\
& \leq \e e^{-r(t)} |\sigma_d|^2\d t + 2\sqrt{\e} e^{-r(t)} (\sigma_d, \we(t)) \d W(t)\nonumber \\
& \quad + e^{-r(t)}\int_Z\varepsilon|g_d|^2N(dt,dz) +
2e^{-r(t)}\varepsilon\int_Z\big(\we(t),g_d\big)\tilde{N}(dt,dz).
\end{align*}
Now integrating from $0\leq t\leq T$ and taking the expectation on both sides and noting that $ \varepsilon < \frac{\nu}{L}$. Also using the fact that
$$2\sqrt{\e}\int_0^T e^{-r(t)} (\sigma_d, \we(t)) \d W(t)$$ and $$2\int_0^Te^{-r(t)}\varepsilon\int_Z\big(\we(t),g_d\big)\tilde{N}(dt,dz)$$
are martingales having zero averages, we get
\begin{align*}
&\mathbb{E}\left[e^{-r(t)}|\we(t)|^2\right] + \nu\mathbb{E}\int_0^Te^{-r(t)}\|\we(t)\|^2\d t  \\
&\quad \leq \mathbb{E}|w(0)|^2 + \mathbb{E}\int_0^Te^{-r(t)}\e |\sigma_d|^2\d t + \mathbb{E}\int_0^Te^{-r(t)}\int_Z\varepsilon|g_d|^2\lambda(dz)\d t
\end{align*}
Using condition (H.3), one can deduce that
\begin{align*}
\mathbb{E}\left[e^{-r(t)}|\we(t)|^2\right] + \left(\nu-\varepsilon L\right)\int_0^Te^{-r(t)}\|\we(t)\|^2\d t \leq\mathbb{E}|w(0)|^2
\end{align*}
Sine $\varepsilon < \frac{\nu}{L}$, we obtain P-a.s.
\begin{align*}
\mathbb{E}\left[e^{-r(t)}|\we(t)|^2\right] \leq \mathbb{E}|w(0)|^2,
\end{align*}
which assures the uniqueness of the strong solution.
\end{proof}

\begin{corollary}
The existence and uniqueness of the strong solution of the stochastic GOY model
\begin{align*}
&\d \ue + \big[\nu A\ue + B(\ue, \ue)\big] \d t = f(t) \d t + \sqrt{\e}\sigma(t, \ue) \d W(t) +\int_Zg(\ue,z)\tilde{N}(dt,dz)\\
& \ue(0) = u_0,\nonumber
\end{align*}
can be proved similarly for the adapted process $\ue(t,x,\omega)$
with the regularity
$$\ue \in \mathrm{L}^2(\Omega; \mathcal{D}(0, T; H) \cap \mathrm{L}^2(0, T; V))$$ under the hypotheses
\begin{itemize}
\item[A.1.]  The function $\sigma \in C([0, T] \times V; L_Q(H_0; H))$, and $g\in \mathbb{H}^2_{\lambda}([0, T] \times Z; H)$.

\item[A.2.]  For all $t \in (0, T)$, there exists a positive constant $K$ such that for all $u \in H$,
$$\varepsilon|\sigma(t, u)|^2_{L_Q} + \int_{Z} |g(u, z)|^2_{H}\lambda(dz) \leq K(1 +\|u\|^2).$$

\item[A.3.]  For all $t \in (0, T)$,  there exists a positive constant $L$ such that for all $\textbf{u}, \textbf{v} \in H$,
$$\varepsilon|\sigma(t, u) - \sigma(t, v)|^2_{L_Q} + \int_{Z} |g(u, z)-g(v, z)|^2_{H}\lambda(dz)\leq L\|u - v\|^2.$$
\end{itemize}
\end{corollary}

\begin{corollary}
Sabra shell model of turbulence is the other well accepted model in the literature, and the fundamental difference with the GOY model lies in the number of complex conjugation operators used in the nonlinear terms which are
responsible for differences in the phase symmetries of the two
models, and as a consequence, Sabra shell  model exhibits
shorter-ranged correlations than the GOY model (see L'vov et. al.
\cite{Lv}). The equations of motion of the stochastic Sabra shell model have the following form
\begin{align*}
\frac{\d u_n}{\d t} + \nu k_n^2 u_n &+ i\big(a k_{n+1} u_{n+2}u\s_{n+1} + b k_n u_{n+1}u\s_{n-1} - \nonumber\\ & -ck_{n-1} u_{n-1}u_{n-2}\big)
= f_n, \quad\text{for}\ n= 1, 2, \ldots,
\end{align*}
along with the boundary conditions
\begin{equation*}
u_{-1} = u_0 = 0.
\end{equation*}
One can deduce from the above equation in the continuous setting with L\'{e}vy noise as
\begin{align*}
&\d \ue + \big[\nu A\ue + B(\ue, \ue)\big] \d t = f(t) \d t + \sqrt{\e}\sigma(t, \ue) \d W(t) + \e\int_Zg(\ue,z)\tilde{N}(dt,dz)\\
& \ue(0) = u_0,\nonumber
\end{align*}
Under the hypothesis \ref{hyp}, and under the same functional
setting, the existence and uniqueness of the strong solution can be
established in $\mathrm{L}^2(\Omega; \mathcal{D}(0, T;
H)\cap\mathrm{L}^2(0, T; V))$.
\end{corollary}

\medskip\noindent
{\bf Acknowledgements:} The second author would like to thank
Council of Scientific and Industrial Research(CSIR) for a Junior
Research Fellowship.


\begin{thebibliography}{99}

\bibitem{Ac1} {\sc de Acosta, A.} (2000). A general non-convex large
    deviation result with applications to stochastic equations; {\em
    Prob. Theory Relat. Fields}, {\bf 118}, 483--521.

\bibitem{Ac} {\sc de Acosta, A.} (1994). Large deviations for vector
    valued L\'{e}vy processes; {\em Stochastic Processes and their
    Applications}, {\bf 51}, 75--115.

\bibitem{Al} {\sc Albeverio, S., Wu, J.L., and Zhang, T.S.} (1998).
    Parabolic SPDEs driven by Poisson white noise; {\em Stochastic
    Processes and their Applications}, {\bf 74}, 21--36.


\bibitem{ARW} {\sc Albeverio, S., R\"{u}diger, B., and Wu, J.L.}
(2001). Analytic and probabilistic aspects of L\'{e}vy processes and
fields in quantum theory; {\em L\'{e}vy Processes: Theory and
Applications}, Barndorff-Nielsen, O., Mikosch, T., and Resnick,
S.I., eds., Birkh\"{a}user Verlag, Basel.

\bibitem{Ap} {\sc Applebaum, D.} (2004).
{\em L\'{e}vy Processes and Stochastic Calculus}, Cambridge Studies
in Advanced Mathematics, Vol. 93, Cambridge University press.

\bibitem{Ba} {\sc Barbato, D., Barsanti, M., Bessaih, H., and Flandoli, F.} (2006). Some rigorous results on a stochastic Goy model;
{\em Journal of Statistical Physics}, {\bf 125}(3), 677--716.

\bibitem{Be} {\sc Bensoussan, A., and Temam, R.} (1972). Equations
aux d\'eriv\'ees partielles stochastiques non lin\'earies(1); {\em
Isr. J. Math.}, {\bf 11}(1), 95--129.





\bibitem{CLT} {\sc Constantin, P., Levant, B., and Titi, E.S.} (2006).
Analytic study of shell models of turbulence; {\em Phys. D}, {\bf
219}(2), 120--141.

\bibitem{DaZ} {\sc Da Prato, G., and Zabczyk, J.} (1992).
{\em Stochastic Equations in Infinite Dimensions}, Cambridge
University Press.








\bibitem{Fo} {\sc Fournier, N.} (2000).
Malliavin calculus for parabolic SPDEs with jumps; {\em Stochastic
Processes and their Applications}, {\bf 87}(1), 115--147.


\bibitem{Fu} {\sc Frisch, U.} (1995).
{\em Turbulence}, Cambridge University Press, Cambridge.

\bibitem{Ha} {\sc Hausenblas, E.} (2005).
Existence, uniqueness and regularity of parabolic SPDEs driven by
Poisson random measure; {\em  Electron. J. Probab.}, {\bf 10},
1496--1546.

\bibitem{IW} {\sc Ikeda, N., and Watanabe, S.} (1989).
{\em Stochastic differential equations and diffusion processes},
North-Holland /Kodansha, Amsterdam, Oxford, New York.


\bibitem{Ka} {\sc Kadanoff, L., Lohse, D., Wang, J., and Benzi, R.} (1995).
Scaling and dissipation in the GOY shell model; {\em Phys. Fluids},
{\bf 7}(3), 617--629.


\bibitem{KS} {\sc Karatzas, I. and Shreve, S.} (1991). {\em Brownian
Motion and Stochastic Calculus}, 2nd edition, Springer-Verlag, New
York.


\bibitem{La} {\sc Ladyzhenskaya, O.A.} (1969).
{\em The Mathematical Theory of Viscous Incompressible Flow}, Gordon
and Breach, New York.

\bibitem{Lv} {\sc L'vov, V.S., Podivilov, E., Pomyalov, A.,
Procaccia, I., and Vandembroucq, D.} (1998). Improved shell model of
turbulence; {\em Phys. Rev. E(3)}, {\bf 58}(2), 1811--1822.

\bibitem{MR} {\sc Mandrekar, V., and R\"{u}diger, B.} (2006).
Existence and uniqueness of pathwise solutions for stochastic
integral equations driven by L\'{e}vy noise on separable Banach
spaces; {\em Stochastics}, {\bf 78}(4), 189--212.

\bibitem{Ma1} {\sc Manna, U., Sritharan, S.S. and Sundar, P.} (2009). Large Deviations for the Stochastic Shell Model of Turbulence; \emph{Nonlinear Differential Equations and Applications (NoDEA)}, {\bf 16}, 493-521.




\bibitem{Ms} {\sc Menaldi, J.L. and Sritharan, S.S.}(2002). Stochastic $2$-D
Navier-Stokes Equation; {\em Appl. Math. Optim.}, {\bf 46}, 31--53.

\bibitem{Me} {\sc Metivier, M.} (1988).
 {\em Stochastic Partial Differential Equations in Infinite Dimensional Spaces},
 Quaderni, Scuola Normale Superiore, Pisa.

\bibitem{Mu} {\sc Mueller, C.} (1998).
The heat equation with L\'{e}vy noise; {\em Stochastic Processes and
their Applications}, {\bf 74}(1), 67--82.

\bibitem{OY} {\sc Ohkitani, K., and Yamada, M.} (1989).
Temporal intermittency in the energy cascade process and local
Lyapunov analysis in fully developed model of turbulence;
 {\em Prog. Theor. Phys.}, {\bf 89}, 329--341.



\bibitem{PZ} {\sc Peszat, S., and Zabczyk, J.} (2007).
{\em Stochastic Partial Differential Equations with L\'{e}vy Noise},
Encyclopedia of Mathematics and Its Applications 113, Cambridge
University Press


\bibitem{RZ} {\sc R\"{o}ckner, M., and Zhang, T.S.} (2007).
Stochastic Evolution Equations of Jump Type: Existence, Uniqueness
and Large Deviation Principles; {\em Potential Analysis}, {\bf 26},
255--279.

\bibitem{Ru} {\sc R\"{u}diger, B.} (2004).
Stochastic Integration with respect to compensated Piosson randon
measures on separable Banach spaces; {\em Stochastics and
Stochastics Reports}, {\bf 76}(3), 213-242.

\bibitem{SZF} {\sc Shlesinger, M.F., Zavslavsky, G.M., and Feisch,
U.} eds., (1995). {\em L\'{e}vy Flights and Related Topics in
Physics}, Springer-Verlag.





\bibitem{Te} {\sc Temam, R.} (1984).
{\em Navier-Stokes Equations, Theory and Numerical Analysis},
North-Holland, Amsterdam.


\bibitem{VF} {\sc Vishik, M. J. and Fursikov, A. V.} (1980). {\em Mathematical Problems of Statistical
Hydromechanics}, Kluwer Academic Press, Boston.

\bibitem{We} {\sc West, B.J.} (1985).
{\em An Essay of the Importance Being Non Linear}, Lecture Notes in
Biomathematics 62, Springer-Verlag, Berlin.

\bibitem{ZY} {\sc Zhao, D. and Chao, X.Y.} (2009). Global solutions of stochastic $2$D Navier-Stokes equations with L\'{e}vy noise; {\em Sci China Ser A}, {\bf 52}(7), 1497--1524.


\end{thebibliography}
\end{document}